\documentclass[11,a4paper]{article}
\usepackage[english]{babel}
\usepackage[utf8]{inputenc}
\usepackage{amsmath,amssymb,graphicx,enumerate,latexsym,theorem,mathtools}
\usepackage{color}
\usepackage[
  hidelinks,
  letterpaper,
  pagebackref,
  bookmarksopen,
  bookmarksnumbered,
]{hyperref}
\hypersetup{
  pdfauthor={Miguel de Benito Delgado},
  pdfauthor={Bernd Schmidt}
  % ...
}
\catcode`\<=\active \def<{
\fontencoding{T1}\selectfont\symbol{60}\fontencoding{\encodingdefault}}
\newcommand{\Iota}{\mathrm{I}}
\newcommand{\TeXmacs}{T\kern-.1667em\lower.5ex\hbox{E}\kern-.125emX\kern-.1em\lower.5ex\hbox{\textsc{m\kern-.05ema\kern-.125emc\kern-.05ems}}}
\newcommand{\assign}{\coloneqq}
\newcommand{\backassign}{=:}
\newcommand{\chapter}[1]{\medskip\bigskip

\noindent\textbf{\huge #1}}
\newcommand{\equallim}{\mathop{=}\limits}
\newcommand{\mathd}{\,\mathrm{d}}
\newcommand{\mathe}{\mathrm{e}}

\newcommand{\nobracket}{}
\newcommand{\nocomma}{}
\newcommand{\nospace}{}
\newcommand{\tmdfn}[1]{\textbf{#1}}
\newcommand{\tmem}[1]{{\em #1\/}}
\newcommand{\tmop}[1]{\ensuremath{\operatorname{#1}}}

\newcommand{\tmrsup}[1]{\textsuperscript{#1}}

\newenvironment{enumeratealpha}{\begin{enumerate}[a{\textup{)}}] }{\end{enumerate}}

\newenvironment{proof}{\noindent\textbf{Proof\ }}{\hspace*{\fill}$\Box$\medskip}

\newenvironment{tmparmod}[3]{\begin{list}{}{\setlength{\topsep}{0pt}\setlength{\leftmargin}{#1}\setlength{\rightmargin}{#2}\setlength{\parindent}{#3}\setlength{\listparindent}{\parindent}\setlength{\itemindent}{\parindent}\setlength{\parsep}{\parskip}} \item[]}{\end{list}}

\newtheorem{corollary}{Corollary}
\newtheorem{definition}{Definition}
\newtheorem{assumptions}{Assumption}
\newtheorem{lemma}{Lemma}
\newcounter{nnnotation}

\newtheorem{notation*}[nnnotation]{Notation}
\newcounter{nnnote}

{\theorembodyfont{\rmfamily}\newtheorem{note*}[nnnote]{Note}}
{\theorembodyfont{\rmfamily}\newtheorem{remark}{Remark}}
\newtheorem{theorem}{Theorem}
\newenvironment{step}[2]{\underline{Step {#1}}

{#2}}

\newcommand{\grs}{\ensuremath{\nabla_{s}}}
\newcommand{\gra}{\ensuremath{\nabla_{a}}}

\newcommand{\applicationspace}[1]{\quad}

\begin{document}

\begin{center}
\begin{LARGE}
A hierarchy of multilayered plate models
\end{LARGE}
\\[0.5cm]
\begin{large}
Miguel de Benito Delgado\footnote{Universit{\"a}t Augsburg, Germany, {\tt m.debenito.d@gmail.com}} and 
Bernd Schmidt\footnote{Universit{\"a}t Augsburg, Germany, {\tt bernd.schmidt@math.uni-augsburg.de}}
\end{large}
\\[0.5cm]
\today
\\[1cm]
\end{center}

%\subjclass[2010]{XXX, XXX, XXX}
%\keywords{XXX, XXX, XXX} 
 
\begin{abstract}
We derive a hierarchy of plate theories for heterogeneous multilayers from three dimensional nonlinear elasticity by means of $\Gamma$-convergence. We allow for layers composed of different materials whose constitutive assumptions may vary significantly in the small film direction and which also may have a (small) pre-stress. By computing the $\Gamma$-limits in the energy regimes in which the scaling of the pre-stress is non-trivial, we arrive at linearised Kirchhoff, von K{\'a}rm{\'a}n, and fully linear plate theories, respectively, which contain an additional spontaneous curvature tensor. The effective (homogenised) elastic constants of the plates will turn out to be given in terms of the moments of the pointwise elastic constants of the materials.
\end{abstract}

\tableofcontents

%-----------------------------------------------------------------------------------------
%-----------------------------------------------------------------------------------------
%-----------------------------------------------------------------------------------------
\section{Introduction}\label{sec:intro}
%-----------------------------------------------------------------------------------------

The derivation of effective theories for thin structures such as beams, rods, plates and shells is a classical problem in continuum mechanics. Fundamental results in formulating adequate dimensionally reduced theories for three-dimensional elastic objects have already been obtained by Euler \cite{Euler}, Kirchhoff \cite{Kirchhoff} and von K{\'a}rm{\'a}n \cite{vonKarman}, cf.\ also \cite{Love,ciarlet_mathematical_1997,ciarlet_mathematical_2000}. 

A physical plate, given by a  domain $\Omega_h = \omega \times (- h / 2, h / 2) \subset \mathbb{R}^3$, is identified with a hyperelastic body of height $h$ ``much smaller'' than the lengths of the sides of $\omega$. The plane domain $\omega \subset \mathbb{R}^2$ constitutes the {\em mid-layer} of the plate. We assume that the body has a (possibly non-homogeneous)
stored energy density $W$ (precise conditions on $W$ will be specified later)
and, after deformation by $\tilde{y} : \Omega_h \to \mathbb{R}^3$, the total {\em elastic energy} 
\[ E_h (\tilde{y}) =
   \int_{\Omega_h} W (z, \nabla \tilde{y} (z)) \mathd z. \] 
The problem amounts to identifying {\em effective} functionals in the limit $h \to 0$ operating on dimensionally reduced deformations of the mid-plane. In spite of its long history, rigorous results in this direction relating classical models for plates to the parent three-dimensional elasticity theory have only been obtained comparably recently. 

In order to avoid working on 
a changing domain, a rescaling $x_3 = z_3 / h$ is performed to obtain a fixed
$\Omega_1$. We set $z_h (x_1, x_2, x_3) = (x_1, x_2, hx_3)$ and we consider
instead of a deformation $\tilde{y} : \Omega_h \rightarrow \mathbb{R}^3$, the
{\em rescaled} one $y_h : \Omega_1 \rightarrow \mathbb{R}^3, y_h (x) =
\tilde{y} (z_h (x))$. We define the 
{\em energy per unit volume} as $J_h = \frac{1}{h} E_h$, which after a
change of variables can be seen to be
\[ J_h (y) = \int_{\Omega_1} W (x, \nabla_h y) \mathd x, \]
where $\nabla_h = (\partial_1, \partial_2, \partial_3 / h)$.

After first results in linear elasticity had been established, see 
{\cite{acerbi_thin_1988,anzellotti_dimension_1994}}, perhaps the first 
work to derive a non-linearly elastic, lower dimensional
theory with a rigorous analysis using variational convergence was
{\cite{acerbi_variational_1991}} for the case of strings. In the context of
nonlinear plates we consider the rescaled functionals 
\[ J^{\beta}_h (y) = \frac{1}{h^{\beta}}  \int_{\Omega_1} W (\nabla_h y) . \]

For {$\beta = 0$}, inspired by the work in {\cite{acerbi_variational_1991}}, 
in {\cite{ledret_nonlinear_1995}} a {\em non-linear membrane theory} is derived. 
The range $\beta \in (0, 1)$ is the so-called {\em constrained membrane} regime, 
analysed in detail in {\cite{conti_lowenergy_2004}}. To the best of 
our knowledge, the regime $\beta \in [1, 2)$ remains not very well 
explored, except under certain kinds of boundary
conditions or assumed admissible deformations, see, e.g., 
{\cite{belgacem_energy_2002}} and the work in {\cite{conti_confining_2008}}. 

Most significant in view of our setup are the contributions to the cases 
$\beta \ge 2$. In {\cite{friesecke_theorem_2002}} Friesecke, James 
and M{\"u}ller prove the fundamental
{\em geometric rigidity} estimate which carries Korn's inequality to
the nonlinear setting and utilise it to obtain the non-linear Kirchhoff theory
of pure bending under an isometry constraint in case {$\beta = 2$}. 
This estimate is at the
core of most of the later developments in this area. In their seminal paper
{\cite{friesecke_hierarchy_2006}}, the same authors exploit the quantitative
geometric rigidity estimate of {\cite{friesecke_theorem_2002}} in a systematic
investigation of limits for the whole range of scalings $\beta \in [2,
\infty)$, deriving the first {\em hierarchy} of limit models. They also
provide a thorough (albeit succinct) overview of the state of the art around
2006. The lecture {\cite[Chapter 2]{dacorogna_vectorvalued_2017}} 
provides a nice waltkthrough of this paper, as well as abundant
references and open problems as of 2017.

This variational approach has been extended and revisited in a variety of 
different contexts, among them more complex shell geometries {\cite{fjmm}}, 
more basic atomistic models {\cite{Schmidt:06,BraunSchmidt:19}}, or more 
complicated material properties as incompressibility 
{\cite{ContiDolzmann:09}}, brittleness {\cite{Schmidt:17}} or oscillatory 
dependence on the space variable 
{\cite{NeukammVelvic13,HornungNeukammVelvic14,HornungPawelczykVelvic18}}. 
Moreover, 
the convergence of equilibria and even dynamic solutions have been established 
{\cite{MuellerPakzad:08,MoraMuellerSchultz:07,AbelsMoraMueller:11}}. 

The focus in this contribution is on materials whose reference configuration is
subjected to stresses (one speaks of {\em pre-strained} or
{\em pre-stressed} bodies) and whose energy density exhibits a dependence
on the out-of-plane direction (modelling {\em multilayered} plates).
Examples of these situations are heated materials, crystallisations on top of
a substrate and multilayered plates. 

For $\beta = 2$ the second author derived in 
{\cite{schmidt_minimal_2007,schmidt_plate_2007}} an effective Kirchhoff theory 
for stored energy densities of the form $W (x_3, F) =
W_0 (x_3, F (I + hB^h (x_3)))$, depending explicitly on the out-of-plane 
coordinate $x_3$ and a ``mismatch tensor'' $B^h(x_3)$ which measures the deviation 
of the energy well $\operatorname{argmin} W(x_3, \cdot)$ from the rigid motions 
$\operatorname{argmin} W_0(x_3, \cdot) = \mathrm{SO}(3)$. We remark that the regime 
$\beta = 2$ is precisely adapted to capture the effects of a misfit 
$hB^h$ scaling linearly $h$. In the simplest case with linearly changing 
$B^h(x_3) = a x_3 \tmop{Id}$  
one obtains a $\Gamma$-limit $I_{\rm Ki}$ with 
\[ I_{\rm Ki} (y) \equallim  \frac{1}{24}  \int_{\omega} Q
   (\tmop{II} - a_1 \tmop{Id}) - a_2 \, \mathd x,  \]
if $y \in \mathcal{A}$ (and $I_{\rm Ki}(y) = + \infty$ if not), where $\mathcal{A}$ 
is a suitable class of admissible deformations (isometric 
immersions). $Q$ is a quadratic form acting on the shape tensor $\tmop{II}$ 
(the second fundamental form of $y$). The coefficients of $Q$ and the numbers 
$a_1, a_2$ can be explicitly computed. In {\cite{schmidt_minimal_2007,schmidt_plate_2007}} 
also a thorough investigation of the shape of energy minimisers (for free boundary 
conditions) is provided which shows that the optimal configurations are 
rolled-up portions of cylinders whose winding direction is determined by 
the material parameters and the misfit tensor. 

The main goal of our work is to extend such an analysis to the energy regimes 
$\beta > 2$ in order to allow for more general pre-strain scalings of the form 
$h^{\alpha-1} B^h (x_3)$, $\alpha > 2$. A main source of  motivation are physical 
experiments which show that there are situations in which optimal configurations 
are spherical caps (paraboloids with positive Gau{\ss} curvature) rather than cylinders, \cite{SalamonMasters93,SalamonMasters95,FinotSuresh96,Freund00,KimLombardo08,EgunovKorvinkLuchnikov16}. 
We will see that indeed this discrepancy can be explained in terms of different 
energy scaling regimes, where the von K{\'a}rm{\'a}n scaling $\beta = 4$ is critical. 
In the present paper we lay the foundation for this by deriving effective plate theories 
for pre-strained multilayers. We analyse the functionals obtained here in depth in our companion paper {\cite{DeBenitoSchmidt:19b}}. 

Indeed there are previous results for $\beta = 4$ in particular. With the aim to model, e.g., growth processes in plants, in {\cite{lewicka_fopplvon_2011}}
the authors derive the von Kármán functional with a spontaneous curvature term for pre-stressed plates. However, their setup is not comparable to our situation. On the one hand, it is even more general as an explicit $(x_1, x_2)$ dependence of the misfit is allowed. On the other hand, there is no explicit $x_3$ dependence as would be necessary to model multilayers. Very recently, these results have been extended to other scalings and significantly $x_3$-dependent misfits, see {\cite{LewickaLucic18}}. However, the treatment of energy densities which may vary considerably in the thin film direction is, as we will see, subtle. A main source of technical difficulties is the fact that in our situation we can no longer expect the mid plane to follow the limiting plate deformation exactly. This phenomenon can be observed already in the simplest situation of a bilayer with one layer being much softer than the other. If rolled up, the unstretched plane will move into the stiffer layer, to an extent which depends on the local curvature. Moreover, we introduce an additional fine scale $\theta$ at a critical exponent, cf.\ below. 

Yet it turns out that in our setup the von Kármán case $\beta = 4$ is in fact a rather straightforward extension of {\cite{friesecke_hierarchy_2006,schmidt_plate_2007}}. The regime $\beta > 4$ is however a bit more involved. In contrast to the homogeneous case in {\cite{friesecke_hierarchy_2006}}, the dependence on the in-plane variable may be non-trivial so it cannot be discarded by setting it to $0$ without loss of generality. The scaling in the linearised Kirchhoff case $\beta \in (2,4)$ turns out to be the most difficult. In order to construct recovery sequences we need to provide a representation result for symmetric tensor fields on $\omega$ in terms of symmetrised gradients and solutions to the non-elliptic Monge-Ampere equation $\det \nabla^2 v = 0$, cf. Theorem~\ref{thm:representation-matrix}. In all cases the resulting effective functionals are explicitly computed with homogenised material constants that can be calculated from the first moments in $x_3$ of the individual elasticity constants of the various layers. 

From a modelling point of view, a main novelty is our introducing a new interpolating regime in between the linearised Kirchhoff case $\beta < 4$ and the fully linear case $\beta > 4$. This is motivated by our findings in \cite{DeBenitoSchmidt:19b} which show that minimisers (after rescaling) coincide for all $\beta \in (2, 4)$ (parts of a cylinder) and for all $\beta \in (4, \infty)$ (parts of a parabolic cap). We introduce a new scaling regime $\theta h^4$ with $\theta \in (0, \infty)$ and obtain von Kármán functionals that upon varying $\theta$ continuously connect the extreme cases $\theta \to 0$ and $\theta \to \infty$, which turn out to reduce to the functionals obtained for $\beta > 4$ and $\beta < 4$, respectively. In the simplest non-trivial example, the prototypical limit functional is of von Kármán type:
\begin{equation}
   \label{eq:simple-interpolating-functional}
    \mathcal{I}^{\theta}_{\rm vK} (u, v) = \frac{\theta}{2}  \int_{\omega}
    Q_2 (\grs u + \tfrac{1}{2} \nabla v \otimes \nabla v) \mathd x + \frac{1}{24}
    \int_{\omega} Q_2 (\nabla^2 v - I) \mathd x.
\end{equation}
In contrast to the cases $\beta \ne 4$ minimisers of this functional are not explicit. We discuss their behaviour in detail in \cite{DeBenitoSchmidt:19b}, in particular, how they interpolate in between $\beta < 4$ and $\beta > 4$.

\subsection*{Outline}

Having fixed the precise setup in Section \ref{sec:the-setting}, in Section 
\ref{sec:main-results} we present our main results: 
Theorem \ref{thm:gamma-hierarchy} on $\Gamma$-convergence in a hierarchy of 
energy scalings and Theorem \ref{thm:gamma-interpolating} on the asymptotic 
behaviour of the interpolating von Kármán functional for $\theta \rightarrow 0$ or 
$\theta \rightarrow \infty$. We then recall some basic results on compactness 
and explicit representations for the limit strains from 
{\cite{friesecke_hierarchy_2006}} in Section \ref{sec:compactness-hierarchy}. 
Proofs of lower and upper bounds in 
Theorem \ref{thm:gamma-hierarchy} are collected in Section 
\ref{sec:gamma-convergence-hierarchy}, where we obtain 
{\eqref{eq:simple-interpolating-functional}} and more general functionals. In
Section \ref{sec:gamma-convergence-interpolation} we show how the von Kármán
functional interpolates between different theories. Finally,
in Section \ref{sec:approximation-and-representation} we prove some density
and matrix representation theorems essential for the construction of recovery
sequences and identification of minimisers in the linearised Kirchhoff regime.

\subsection*{Notation}

We denote by $e_1, e_2, e_3$ the standard basis vectors in $\mathbb{R}^3$ and 
write $x = (x', x_3) \in \mathbb{R}^3, x' \in \mathbb{R}^2$. The spaces of 
symmetric and antisymmetric $n \times n$ matrices are 
$\mathbb{R}^{n \times n}_{\tmop{sym}}$ and $\mathbb{R}^{n \times n}_{\tmop{ant}}$, 
respectively. $A_{\tmop{sym}} = \tmop{sym} A 
= \tfrac{1}{2}  (A + A^{\top})$ is the symmetric part and $A_{\tmop{ant}} = 
\tmop{ant} A = \tfrac{1}{2}  (A - A^{\top})$ the antisymmetric part of a 
square matrix $A$.

Attaching a row and a column of zeros to a matrix $G \in \mathbb{R}^{2 \times 2}$ 
leads to $\hat{G} \assign \sum_{\alpha,\beta = 1}^2 G_{\alpha \nocomma \beta} 
e_{\alpha} \otimes e_{\beta} \in \mathbb{R}^{3 \times 3}$, conversely, 
$\check{B} \in \mathbb{R}^{2 \times 2}$ 
is the matrix resulting from the deletion of the third row and column of any 
$B \in \mathbb{R}^{3 \times 3}$. If $Q(\cdot)$ is a quadratic form, we denote 
the associated bilinear form by $Q[ \cdot, \cdot]$.

For $f : \mathbb{R}^3 \rightarrow \mathbb{R}$ a scalar function 
$\nabla f = (\partial_1 f, \partial_2 f, \partial_3 f)^{\top}$ is a column 
vector, whereas for $y : \mathbb{R}^3 \rightarrow \mathbb{R}^3$ we have 
$\nabla y \in \mathbb{R}^{3 \times 3}$ with rows $\nabla^{\top} y_i$, i.e.,  
$(\nabla y)_{i \nocomma j} = y_{i, j} 
= \partial_j y_i, i, j \in \{ 1, 2, 3 \}$. Its left $3 \times 2$ 
submatrix is $\nabla' y$, its rescaled gradient $\nabla_h y 
= ( \partial_1 y, \partial_2 y, 1 / h \, \partial_3 y )$. Moreover, 
$\grs u = \frac{1}{2} (\nabla u + \nabla^{\top} u)$, is the symmetrised 
gradient of $u : \mathbb{R}^2 \rightarrow \mathbb{R}^2$, $\nabla^2 v$ 
the Hessian matrix of $v : \mathbb{R}^n \rightarrow \mathbb{R}$.   

Let $\omega \subset \mathbb{R}^2$. We set 
$\hat{\nabla} v \assign (\partial_1 v, \partial_2 v, 0)^{\top} \in
    \mathbb{R}^3$ for $v : \omega \rightarrow \mathbb{R}$, 
$\hat{\nabla} u \assign \sum_{\alpha,\beta = 1}^2 (\nabla' u)_{\alpha \nocomma \beta} e_{\alpha}
    \otimes e_{\beta} \in \mathbb{R}^{3 \times 3}$ for $u : \omega \rightarrow \mathbb{R}^2$ and 
$\hat{\nabla} b \assign \sum_{\alpha = 1}^3 \sum_{\beta = 1}^2 (\nabla' b)_{\alpha \nocomma \beta} e_{\alpha}
    \otimes e_{\beta} \in \mathbb{R}^{3 \times 3}$ for $b : \omega \rightarrow \mathbb{R}^3$. 

The norm on Sobolev spaces is $\| v \|_{k, p, \Omega} = \| v \|_{W^{k, p} (\Omega)}$. We will omit the 
domain when it is clear from the context. 

We abbreviate $A_{\theta} \assign \grs u_{\theta} 
+ \tfrac{1}{2} \nabla v_{\theta} \otimes \nabla v_{\theta}$, mostly in 
Section \ref{sec:gamma-convergence-interpolation} and set 
$(f)_{\omega} \assign \frac{1}{| \omega |}  \int_{\omega} 
f (x') \mathd x'$ is the average of $f$ over $\omega$.

%-----------------------------------------------------------------------------------------
%-----------------------------------------------------------------------------------------
%-----------------------------------------------------------------------------------------
\section{The setting}\label{sec:the-setting}
%-----------------------------------------------------------------------------------------

As described in Section \ref{sec:intro}, we consider a sequence of
increasingly thin domains $\Omega_h \assign \omega \times (- h / 2, h / 2) \in
\mathbb{R}^3$ and rescale them to
\[ \Omega_1 \assign \omega \times (- 1 / 2, 1 / 2) \subset \mathbb{R}^3 \]
where $\omega \subset
\mathbb{R}^2$ is bounded with Lipschitz boundary. As a consequence of the rescaling, instead of maps $\tilde{y} : \Omega_h \rightarrow
\mathbb{R}^3$, we consider the {\em rescaled deformations}
\[ y : \Omega_1 \rightarrow \mathbb{R}^3, x \mapsto y (x) = \tilde{y} (x_1, x_2,
   hx_3), \]
belonging to the space
\[ Y \assign W^{1, 2} (\Omega_1 ; \mathbb{R}^3) . \]
For each {\em scaling}\footnote{In the notation of Section \ref{sec:intro} 
we have $\beta = 2 \alpha - 2$ .}
\[ \alpha \in (2, \infty), \]
and for all deformations $y \in Y$, define the {\em scaled elastic energy}
per unit volume:
\begin{equation}
  \label{eq:scaled-elastic-energy} \mathcal{I}_{\alpha}^h (y) = \frac{1}{h^{2
  \alpha - 2}}  \int_{\Omega_1} W_{\alpha}^h (x_3, \nabla_h y (x)) \mathd x,
\end{equation}
where $\nabla_h = (\partial_1, \partial_2, \partial_3 / h)^{\top}$ is the
gradient operator resulting after the change of coordinates described in
Section \ref{sec:intro}. For the sake of conciseness, we will present
most results below for all scalings simultaneously, adding the parameter
$\alpha$ to much of the notation. The energy density for $\alpha \neq 3$ is
given by
\begin{equation*}
%  \label{eq:stored-energy-density} 
  W_{\alpha}^h (x_3, F) = W_0 (x_3, F (I +
  h^{\alpha - 1} B^h (x_3))), \applicationspace{1 \tmop{em}} F \in
  \mathbb{R}^{3 \times 3} .
\end{equation*}
where $B^h : \left( - 1 / 2, 1 / 2 \right) \rightarrow \mathbb{R}^{3 \times
3}$ describes the {\em internal misfit} and $W_0$ the stored energy
density of the reference configuration. In the regime $\alpha = 3$ we include
an additional parameter $\theta > 0$ controlling further the amount of misfit
in the model:\label{ref:introducing-theta}
\[ W_{\alpha = 3}^h (x_3, F) = W_0 \left( x_3, F \left( I + h^2  \sqrt{\theta}
   B^h (x_3) \right) \right), \applicationspace{1 \tmop{em}} F \in
   \mathbb{R}^{3 \times 3}, \]
and we later write $\tilde{B}^h = \sqrt{\theta} B^h$. Note that given the
choice $h^{\alpha - 1}$ for the scaling of the misfit, the fact that in the
limit it will be again scaled quadratically forces the choice of a scaling of
$h^{- 2 (\alpha - 1)}$ for the energy, since otherwise one would compute
trivial (vanishing or infinite) energies in the limits. This will become
apparent in the computation of the lower bounds in Theorem
\ref{thm:lower-bound}. Our assumptions for $B^h$ and $W_0$ are those of
{\cite[Assumption 1.1]{schmidt_plate_2007}}:

{\assumptions{\label{main-assumptions}}
\ \\ \vspace{-2ex}
\begin{enumeratealpha}
  \item %\label{assumption:W0-smooth}
  For a.e. $t \in \left( - 1 / 2, 1 / 2
  \right)$, $W_0 (t, \cdot)$ is continuous on $\mathbb{R}^{3 \times 3}$ and
  $C^2$ in a neighbourhood of $\tmop{SO} (3)$ which does not depend on $t$.
  
%  \item For all $F \in \mathbb{R}^{3 \times 3}$, $W_0 (\cdot, F) \in
%  L^{\infty} \left( \left( - 1 / 2, 1 / 2 \right) ; \mathbb{R} \right)$. 
  
  \item \label{assumption:Q3-ess-bounded}The quadratic form $Q_3 (t, \cdot) =
  D^2 W_0 (t, I) [\cdot, \cdot]$ is in $L^{\infty} \left( \left( - 1 / 2, 1 /
  2 \right) ; \mathbb{R}^{9 \times 9} \right)$.
  
  \item \label{assumption:modulus}The map
  \[ \omega (s) \assign \underset{- 1 / 2 < t < 1 / 2}{\tmop{ess} \sup} 
     \underset{| F | \leqslant s}{\sup}  |W_0 (t, I + F) - \tfrac{1}{2} Q_3
     (t, F) | \]
  shall satisfy $\omega (s) = o (s^2)$ as $s \rightarrow 0$.
  
  \item \label{property:frame invariance}For all $F \in \mathbb{R}^{3 \times
  3}$ and all $R \in \tmop{SO} (3)$
  \[ W_0 (t, F) = W_0 (t, RF) . \]

  \item \label{property:energy-well-growth-below}For a.e. $t \in \left( - 1 /
  2, 1 / 2 \right), W_0 (t, F) = 0$ if $F \in \tmop{SO} (3)$ and
  \[ \underset{- 1 / 2 < t < 1 / 2}{\tmop{ess} \inf} W_0 (t, F) \geqslant C
     \tmop{dist}^2  (F, \tmop{SO} (3)), \]
  for all $F \in \mathbb{R}^{3 \times 3}$ and some $C > 0$.
  
  \item $B^h \rightarrow B$ in $L^{\infty} \left( \left( - 1 / 2, 1 / 2
  \right) ; \mathbb{R}^{3 \times 3} \right)$.
\end{enumeratealpha}}

The Hessian
\[ Q_3 (t, F) \assign D^2 W_0 (t, I) [F, F] = \frac{\partial^2 W_0 (t,
   I)}{\partial F_{i \nocomma j} \partial F_{i \nocomma j}} F_{i \nocomma j} F
   \nocomma_{i \nocomma j}, \]
for $t \in \left( - 1 / 2, 1 / 2 \right), F \in \mathbb{R}^{3 \times 3}$ is
twice the quadratic form of linear elasticity theory, which results after a
linearisation of $W_0$ around the identity. By Assumption 
\ref{main-assumptions}.\ref{property:energy-well-growth-below} it is positive 
definite on symmetric matrices and vanishing on antisymmetric matrices. 
We note in passing two consequences of the above conditions. First, frame
invariance (Assumption \ref{main-assumptions}.\ref{property:frame invariance})
extends to the second derivative where defined, i.e.\ 
\[ D^2 W_0 (t, R) [FR, FR] = D^2 W_0 (t, I) [F, F] = Q_3 (t, F) . \]
Second, the energy $W_0$ grows at most quadratically in a neighbourhood of
$\tmop{SO} (3)$, i.e. for small $| F |$ it holds that: 
\begin{equation*}
  %\label{eq:quadratic-growth} 
  W_0 (t, I + F) \leqslant C \tmop{dist}^2 (I + F,
  \tmop{SO} (3)) .
\end{equation*}

Define $Q_2$ to be
the quadratic form on $\mathbb{R}^{2 \times 2}$ obtained by relaxation of
$Q_3$ among stretches in the $x_3$ direction:
\[ Q_2 (t, G) \assign \underset{c \in \mathbb{R}^3}{\min} Q_3 (t, \hat{G} + c
   \otimes e_3) \text{, for } t \in \left( - 1 / 2, 1 / 2 \right), G \in
   \mathbb{R}^{2 \times 2}, \]
where $e_3 = (0, 0, 1) \in \mathbb{R}^3$. (See the last paragraph of Section 
\ref{sec:intro} for the definition of $\hat{G}$.) This process effectively minimises
away the effect of transversal strain. Solving the minimisation problem yields a 
map $\mathcal{L}: I \times \mathbb{R}^{2 \times 2} \rightarrow \mathbb{R}^3$, 
linear in its second argument, which attains the minimum: 
\begin{equation}\label{eq:mapping-L}
  Q_2 (t, G) = Q_3 (t, \hat{G} +\mathcal{L} (t, G) \otimes e_3) . 
\end{equation}
In particular, also the $Q_2 (t, \cdot)$ are positive definite on symmetric matrices 
and vanishing on antisymmetric matrices. In fact, by Assumption 
\ref{main-assumptions}.\ref{assumption:Q3-ess-bounded} and 
\ref{main-assumptions}.\ref{property:energy-well-growth-below} we have the bounds 
\begin{equation}\label{eq:Q2-L-bounds}
  Q_2 (t, F) \gtrsim | F |^2 ~~\forall \, F \in \mathbb{R}^{2 \times 2}_{\rm sym} 
  \quad\text{ and }\quad
  | \mathcal{L} (t, F) | \lesssim | F | ~~\forall \, F \in \mathbb{R}^{2 \times 2}
\end{equation}
uniformly in $t \in (-1/2,1/2)$. 

For the regimes $\alpha \geqslant 3$, we define the effective form
\begin{equation}
  \label{def:Q-bar} \overline{Q}_2 (E, F) \assign \int_{- 1 / 2}^{1 / 2} Q_2
  (t, E + tF + \check{B} (t)) \mathd t,
\end{equation}
with $E, F \in \mathbb{R}^{2 \times 2}$ (see the last paragraph of Section 
\ref{sec:intro} for the definition of $\check{B}$). For $\alpha \in (2, 3)$ we 
consider its relaxation
\begin{equation}
  \label{eq:q2bar} \overline{Q}^{\star}_2 (F) \assign \underset{E \in
  \mathbb{R}^{2 \times 2}}{\min}  \overline{Q}_2 (E, F) = \underset{E \in
  \mathbb{R}_{\tmop{sym}}^{2 \times 2}}{\min} \int_{- 1 / 2}^{1 / 2} Q_2 (t, E
  + tF + \check{B} (t)) \mathd t.
\end{equation}
For the case $\alpha = 3$, we include an additional parameter $\theta > 0$ as
discussed in page \pageref{ref:introducing-theta} and later write $\tilde{B} =
\sqrt{\theta} B$. Both $\overline{Q}_2$ and $\overline{Q}_2^{\star}$ are 
non-negative quadratic forms (see {\cite{DeBenitoSchmidt:19b}} for formulae 
explicitly relating these to $Q_3(t, \cdot)$, $t \in (- 1 / 2, 1 / 2 )$).

For fixed $\alpha \in (2, \infty)$ we say that a sequence $(y^h)_{h > 0}
\subset Y$ has {\em finite scaled energy} if there exists some constant $C
> 0$ such that
\begin{equation*}
  %\label{def:finite-scaled-energy-lki} 
  \underset{h \rightarrow 0}{\tmop{lsup}}
  \mathcal{I}_{\alpha}^h (y^h) \leqslant C.
\end{equation*}
This definition will be central for many of the arguments below. After some
corrections we will have precompactness of such sequences, thus essentially
proving that the family $\mathcal{I}^h_{\alpha}$ is equicoercive, the
essential condition for the fundamental theorem of $\Gamma$-convergence
showing convergence of minimisers and energies. 
This compactness takes place in adequate target ambient spaces
\[ X_{\alpha} = \left\{\begin{array}{ll}
     W^{1, 2} (\omega ; \mathbb{R}) & \text{if } \alpha \in (2, 3),\\
     W^{1, 2} (\omega ; \mathbb{R}^2) \times W^{1, 2} (\omega ; \mathbb{R}) &
     \text{if } \alpha \geqslant 3,
   \end{array}\right. \]
equipped with the weak topology.\footnote{Because the weak
topology is not 1\tmrsup{st} countable, for $\Gamma$-convergence one argues
that one may consider bounded sets, where it is metrisable.}

An essential ingredient in arguments with $\Gamma$-convergence is the choice
of sequential convergence to obtain (pre-)compactness. For the lower bounds we may suppose that a
sequence $(y^h)_{h > 0}$ has finite scaled energy, which enables Lemma
\ref{lem:approx-rotations} for the identification of the limits. This requires
us to work with the corrected deformations $\rho (y^h) \assign
(\overline{R}^h)^{\top} y^h - \overline{c} ^h$, for some constants
$\overline{R}^h \in \tmop{SO} (3)$ and $\overline{c} ^h \in \mathbb{R}^3$
depending on $y^h$, see {\eqref{eq:corrected-deformations}}.\footnote{These
maps ``remove'' rigid movements from the $y^h$ bringing them close to the
identity. Note that the energy is not affected by this change because of frame
invariance (Assumption \ref{main-assumptions}.\ref{property:frame
invariance}).} We choose to encode this transformation into the definition of
{\tmem{$\Gamma$-convergence via maps $P^h_{\alpha}$}} (Definition
\ref{def:gamma-convergence-via-maps}) for general transformations $\rho$ with
arbitrary $R^h \in \tmop{SO} (3)$ and $c^h \in \mathbb{R}^3$. Despite adding
clutter to the notation, this helps to highlight and isolate the technical
requirement of the sequences involved with special rigid
transformations.\footnote{We only require that there be {\tmem{some}}
constants $R^h, c^h$ for $P^h$-convergence. In order to obtain compactness and
in the lower bounds we will take the specific ones given in Lemma
\ref{lem:approx-rotations} whereas for the recovery sequences we will use $R^h
= I, c^h = 0$.}

\begin{definition}
  \label{def:ph-maps}Let $Y \assign W^{1, 2} (\Omega_1 ; \mathbb{R}^3)$ and
  \[ X_{\alpha} \assign \left\{\begin{array}{lll}
       W^{1, 2} (\omega ; \mathbb{R}) & \text{if} & \alpha \in (2, 3),\\
       W^{1, 2} (\omega ; \mathbb{R}^2) \times W^{1, 2} (\omega ; \mathbb{R})
       & \text{if} & \alpha \geqslant 3.
     \end{array}\right. \]
  We say that a sequence $(y^h)_{h > 0} \subset Y$ {\tmdfn{$P^h$-converges}}
  to some $w \in X_{\alpha}$ if and only if there exist constants $R^h \in
  \tmop{SO} (3), c^h \in \mathbb{R}^3$ which define maps
  \[ \rho : Y \rightarrow Y, y^h \mapsto \rho (y^h) \assign (R^h)^{\top} y^h -
     c^h \]
  such that
  \[ P^h_{\alpha} (y^h) \rightarrow w \text{\quad weakly in } X_{\alpha}, \]
  where
  \[ P^h_{\alpha} : Y \rightarrow X_{\alpha}, y^h \mapsto
     \left\{\begin{array}{ll}
       v_{\alpha}^h, & \text{\quad if } \alpha \in (2, 3),\\
       (u_{\theta}^h, v_{\theta}^h) & \text{\quad if } \alpha = 3,\\
       (u_{\alpha}^h, v_{\alpha}^h), & \text{\quad if } \alpha > 3,
     \end{array}\right. \]
  and we defined:
  
  {\noindent}For $\alpha \neq 3$ and $x' \in \omega$, the
  {\tmdfn{scaled, averaged and corrected in-plane}} and {\tmdfn{out-of-plane
  displacements}}:
  \begin{equation}
    \label{def:scaled-in-out-of-plane-displacements} \left\{
    \begin{array}{lll}
      u_{\alpha}^h (x') & \assign & \frac{1}{h^{\gamma}}  \int_{- 1 / 2}^{1 /
      2} (\rho (y^h)' (x', x_3) - x') \mathd x_3,\\
      v_{\alpha}^h (x') & \assign & \frac{1}{h^{\alpha - 2}}  \int_{- 1 /
      2}^{1 / 2} \rho (y^h)_3 (x', x_3) \mathd x_3,
    \end{array} \right.
  \end{equation}
  where
  \[ \gamma = \left\{\begin{array}{rll}
       2 (\alpha - 2) & \text{if} & \alpha \in (2, 3),\\
       \alpha - 1 & \text{if} & \alpha > 3.
     \end{array}\right. \]
  For $\alpha = 3$ and $x' \in \omega$, introducing the
  additional parameter $\theta > 0$:
  \begin{equation}
    \label{def:scaled-in-out-of-plane-displacements-theta} \left\{
    \begin{array}{lll}
      u_{\theta}^h (x') & \assign & \frac{1}{\theta h^2}  \int_{- 1 / 2}^{1 /
      2} [\rho (y^h)' (x', x_3) - x'] \mathd x_3\\
      v_{\theta}^h (x') & \assign & \frac{1}{\sqrt{\theta} h}  \int_{- 1 /
      2}^{1 / 2} \rho (y^h)_3 (x', x_3) \mathd x_3 .
    \end{array} \right.
  \end{equation}
  For $\alpha = 3$, we overload the notation with the parameter $\theta$
  writing $(u_{\theta}^h, v_{\theta}^h)$ and $P_{\theta}^h$ instead of
  $(u_{\alpha}^h, v_{\alpha}^h)$ or $P^h_{\alpha}$, letting the letter used in
  the subindex resolve ambiguity.
\end{definition}

With Definition \ref{def:ph-maps} we can specify precisely what we mean by
$\Gamma$-convergence of the energies {\eqref{eq:scaled-elastic-energy}}:
\footnote{We refer to the notes {\cite{braides_handbook_2006}} for a quick
introduction to $\Gamma$-convergence.}

\begin{definition}
  \label{def:gamma-convergence-via-maps}Let $\alpha > 2$. We say that the
  family of scaled elastic energies $\{ \mathcal{I}_{\alpha}^h : Y \rightarrow
  \mathbb{R} \}_{h > 0}$, $h > 0$, {\tmdfn{$\Gamma$-converges via maps $P^h$
  to}} $\mathcal{I}_{\alpha} : X_{\alpha} \rightarrow \mathbb{R}$ iff:
  \begin{enumeratealpha}
    \item {\tmdfn{Lower bound:}} For every $w \in X_{\alpha}$ and every
    sequence $(y^h)_{h > 0} \subset Y$ which $P^h$-converges to $w$ as $h
    \rightarrow 0$ it holds that
    \[ \underset{h \rightarrow 0}{\tmop{linf}} \mathcal{I}_{\alpha}^h (y^h)
       \geqslant \mathcal{I}_{\alpha} (w) . \]
    \item {\tmdfn{Upper bound:}} For every $w \in X$ there exists a
    {\tmdfn{recovery sequence}} $(y^h)_{h > 0} \subset Y$ which
    $P^h$-converges to $w$ as $h \rightarrow 0$ and
    \[ \underset{h \rightarrow 0}{\tmop{lsup}} \mathcal{I}_{\alpha}^h (y^h)
       \leqslant \mathcal{I}_{\alpha} (w) . \]
  \end{enumeratealpha}
\end{definition}

Finally, we identify what the space of {\em admissible displacements} for
the limit theories will be:
\[ X^0_{\alpha} \assign \left\{\begin{array}{lll}
     W^{2, 2}_{s \nospace h} (\omega ; \mathbb{R}) & \text{if} & \alpha \in
     (2, 3),\\
     W^{1, 2} (\omega ; \mathbb{R}^2) \times W^{2, 2} (\omega ; \mathbb{R}) &
     \text{if} & \alpha \geqslant 3,
   \end{array}\right. \]
where the space of out-of-plane displacements with singular Hessian
\[ W^{2, 2}_{s \nospace h} (\omega) \assign \left\{ v \in W^{2, 2} (\omega ;
   \mathbb{R}) : \det \nabla^2 v = 0 \text{ a.e.} \right\}, \]
will be central in the linearised Kirchhoff theory. We will define the
functionals to be $+ \infty$ for inadmissible displacements in $X_{\alpha}
\backslash X_{\alpha}^0$.

%-----------------------------------------------------------------------------------------
%-----------------------------------------------------------------------------------------
%-----------------------------------------------------------------------------------------
\section{Main results}\label{sec:main-results}
%-----------------------------------------------------------------------------------------

Our first goal is to prove that in the pre-strained setting described above one
has a hierarchy of plate models à la 
{\cite{friesecke_hierarchy_2006}}. The proof is split into several theorems in
Section \ref{sec:gamma-convergence-hierarchy}. For notation we refer to the end of Section \ref{sec:intro}, for details on our
particular use of $\Gamma$-convergence, see Definition 
\ref{def:gamma-convergence-via-maps}.

\begin{theorem}[Hierarchy of effective theories]
  \label{thm:gamma-hierarchy}Let
  \[ \mathcal{I}_{\alpha}^h (y) = \frac{1}{h^{2 \alpha - 2}}  \int_{\Omega_1}
     W_{\alpha}^h (x_3, \nabla_h y (x)) \mathd x. \]
  If $\alpha \in (2, 3)$ and $\omega$ is convex, then the elastic
  energies $\mathcal{I}_{\alpha}^h$ $\Gamma$-converge to the
  {\em linearised Kirchhoff energy}\footnote{Convexity of the domain is
  required for the representation theorems in Section
  \ref{sec:approximation-and-representation} which are used in the
  construction of the recovery sequence for $\alpha \in (2, 3)$.}
  \begin{equation}
    \label{eq:energy-lki} \mathcal{I}_{\rm lKi} (v) \assign
    \left\{\begin{array}{rl}
      \frac{1}{2}  \int_{\omega} \overline{Q}_2^{\star} (- \nabla^2 v) &
      \text{ if } v \in W^{2, 2}_{s h} (\omega),\\
      \infty & \text{ otherwise},
    \end{array}\right.
  \end{equation}
  where $\overline{Q}^{\star}_2$ is defined in {\eqref{eq:q2bar}}. See
  Theorems \ref{thm:lower-bound} and \ref{thm:upper-bound-lki}.
  
  If $\alpha = 3$ and $\theta > 0$ then the energies $\mathcal{I}_{\theta}^h
  \assign \frac{1}{\theta} \mathcal{I}_{\alpha = 3}^h$ $\Gamma$-converge to
  the {\em von Kármán type energy}\footnote{Again, we slightly overload
  the notation in what would be a double definition of $\mathcal{I}_3^h$,
  trusting the letter used in the subindex to dispel the ambiguity.}
  \begin{equation}
    \label{eq:energy-vk} \mathcal{I}^{\theta}_{\rm vK} (u, v) \assign
    \left\{\begin{array}{l}
      \frac{1}{2}  \int_{\omega} \overline{Q}_2 (\theta^{1 / 2} (\grs u +
      \tfrac{1}{2} \nabla v \otimes \nabla v), - \nabla^2 v)\\
      \text{{\hspace{5em}}if } (u, v) \in W^{1, 2} (\omega ; \mathbb{R}^2)
      \times W^{2, 2} (\omega ; \mathbb{R}),\\
      \infty, \text{ otherwise},
    \end{array}\right.
  \end{equation}
  where $\overline{Q}_2$ is defined in {\eqref{def:Q-bar}}. See Theorems
  \ref{thm:lower-bound} and \ref{thm:upper-bound-vk}.
  
  Finally, if $\alpha > 3$ then $\mathcal{I}_{\alpha}^h$ $\Gamma$-converges to
  the {\em linearised von Kármán energy}
  \begin{equation}
    \label{eq:energy-lvk} \mathcal{I}_{\rm lvK} (u, v) \assign
    \left\{\begin{array}{l}
      \frac{1}{2}  \int_{\omega} \overline{Q}_2 \left( \grs u, - \nabla^2 v
      \right),\\
      \text{{\hspace{5em}}if } (u, v) \in W^{1, 2} (\omega ; \mathbb{R}^2)
      \times W^{2, 2} (\omega ; \mathbb{R})\\
      \infty, \text{ otherwise} .
    \end{array}\right.
  \end{equation}
  See Theorems \ref{thm:lower-bound} and \ref{thm:upper-bound-lvk}. 

  Moreover, in all cases $\alpha > 2$ there exists a subsequence 
  (not relabelled) such hat $(y^h)_{h > 0}$ $P^h$-converges to 
  $v \in X_{\alpha}$ (if $\alpha \in (2,3)$), respectively 
  $(u, v) \in X_{\alpha}$ (if $\alpha \ge 3$), see Lemma 
  \ref{lem:approx-rotations}. 
\end{theorem}

\begin{remark} \ \\ \vspace{-2ex}
\begin{enumerate}
\item We will not be considering body forces for simplicity, but including 
them in the analysis as in {\cite{friesecke_hierarchy_2006}} is straightforward.

\item A standard argument shows that almost minimisers of 
$\mathcal{I}_{\alpha}^h$ $P^h$-converge (up to subsequences) to minimisers of the 
limiting functional $\mathcal{I}_{\rm lKi}$, respectively $\mathcal{I}_{\rm vK}$, 
respectively $\mathcal{I}_{\rm lvK}$. 

\item With the help of elementary computations the effective quadratic forms 
$\overline{Q}^{\star}_2, \overline{Q}_2$ can be rewritten in terms of the 
{\em moments} in $t$ of the individual $Q_3(t, \cdot)$. This is made explicit 
in {\cite{DeBenitoSchmidt:19b}}. 
\end{enumerate}
\end{remark}

The functional $\mathcal{I}_{\rm lKi}$ is said to model a
{\em linearised Kirchhoff} regime because the isometry condition
$\nabla^{\top} y \nabla y = I$ of the Kirchhoff model is replaced by $\det
\nabla^2 v = 0$, a necessary and sufficient condition for the existence of an
in-plane displacement $u$ such that $\nabla u + \nabla^{\top} u + \nabla v
\otimes \nabla v = 0$. This condition is to leading order equivalent to
$\nabla^{\top} y \nabla y = I$ for deformations $y = (h^{2 \alpha - 4} u,
h^{\alpha - 2} v)$.\footnote{In the numerical analysis literature, the
denomination {\em linear Kirchhoff} is sometimes used for a pure bending
regime without constraints.} The functional $\mathcal{I}^{\theta}_{\rm vK}$ 
is of {\em von Kármán} type with in-plane and out-of-plane strains
interacting in a membrane energy term, and a bending energy term. For simple
choices of $Q_2$ and $B^h$, one recovers the classical functional
{\eqref{eq:simple-interpolating-functional}}. Finally, we say that the third
limit $\mathcal{I}_{\rm lvK}$, models a {\em linearised von
Kármán} (or {\em fully linear}) regime by analogy with the classical equivalent, but it is of a
different kind than the one expected from the hierarchy derived in
{\cite{friesecke_hierarchy_2006}}, since it again features an interplay
between in-plane and out-of-plane components.\footnote{This is in contrast to
{\cite{friesecke_hierarchy_2006}}. In our setting with the additional
dependence on the $x_3$ coordinate, it is not possible to simply drop terms
while bounding below the energy in the proof of the lower bound as is done in
{\cite[p. 211]{friesecke_hierarchy_2006}} because of the difficulty in
building recovery sequences later. For $\alpha \in (2, 3)$ we introduce an
additional relaxation and make use of representation Theorem
\ref{thm:representation-matrix} to construct them, but for $\alpha > 3$, no
such result is available. One could think that minimising globally,
$\inf_{\grs u} \int Q_2 \left( t, \grs u + \ldots \right)$, might be a way of
discarding in-plane displacements to recover the standard theory, but this
yields a functional which is not local and therefore lacks an integral
representation (see e.g. {\cite[Chapter 9]{braides_homogenization_1998}}).
Note that even if we pick $Q_2$ independent of $t$ and $B = 0$, we do not
recover the functional of {\cite{friesecke_hierarchy_2006}} because ours keeps
track of both in-plane and out-of-plane displacements which is essential to
capture the effect of pre-stressing with the internal misfit $B^h$.}

Our second goal is to show that the limit energy $\mathcal{I}^{\theta}_{\rm vK}$ 
interpolates between $\mathcal{I}_{\rm lKi}$ and
$\mathcal{I}_{\rm lvK}$ as the parameter $\theta$ moves from
$\infty$ to $0$, so that one can say that the theory of von Kármán type
bridges the other two. More precisely, in Section
\ref{sec:gamma-convergence-interpolation} we prove:

\begin{theorem}[Interpolating regime]
  \label{thm:gamma-interpolating}The following two
  $\Gamma$-limits hold:
  \[ \mathcal{I}_{\rm vK}^{\theta}  \overset{\Gamma}{\underset{\theta
     \uparrow \infty}{\longrightarrow}} \mathcal{I}_{\rm lKi},
  \]
  if $\omega$ is convex (Theorems \ref{thm:lower-bound-vk-to-lki} and
  \ref{thm:upper-bound-vk-to-lki}) and:
  \[ \mathcal{I}_{\rm vK}^{\theta}  \overset{\Gamma}{\underset{\theta
     \downarrow 0}{\longrightarrow}} \mathcal{I}_{\rm lvK} \]
  (Theorems \ref{thm:lower-bound-vk-to-lvk} and
  \ref{thm:upper-bound-vk-to-lvk}). Furthermore, sequences $(u_{\theta},
  v_{\theta})_{\theta > 0}$ of bounded energy $\mathcal{I}^{\theta}_{\rm vK}$ 
  are precompact in suitable spaces as $\theta \uparrow \infty$
  or $\theta \downarrow 0$ (Theorem \ref{thm:compactness-interpolating}).
\end{theorem}

\noindent{\it Example.} 
The easiest non-trivial situation is given by a linear internal misfit in a homogeneous material with 
%\begin{equation}
%  \label{eq:linear-internal-stresses} B (t) \assign tI_3 \in \mathbb{R}^{3
%  \times 3},
%\end{equation}
\[ B (t) \assign tI_3 \in \mathbb{R}^{3  \times 3} 
   \quad \mbox{and} \quad 
   Q_2 (t, \cdot) = Q_2 (\cdot). \] 
Then
\begin{align*}
  \mathcal{I}_{\rm vK}^{\theta} (u, v) 
  &= \frac{\theta}{2}  \int_{\omega} Q_2 (\grs u + \tfrac{1}{2} \nabla v \otimes
  \nabla v) + \frac{1}{24}  \int_{\omega} Q_2 (\nabla^2 v - I). 
\end{align*}
for $(u, v) \in W^{1, 2} (\omega ; \mathbb{R}^2) 
\times W^{2, 2} (\omega ; \mathbb{R})$. We refer to \cite{DeBenito-Thesis} for more 
worked out examples.

%-----------------------------------------------------------------------------------------
%-----------------------------------------------------------------------------------------
%-----------------------------------------------------------------------------------------
\section{Compactness and identification of limit
strain}\label{sec:compactness-hierarchy}
%-----------------------------------------------------------------------------------------

We collect here some basic results proving compactness of sequences of scaled
energy and providing explicit representations for the limit strains, as
required for the proofs of $\Gamma$-convergence in Section 
\ref{sec:main-results}. These results are direct consequences of the homogeneous 
case treated in {\cite[Lemma 1]{friesecke_hierarchy_2006}}. We recall the 
definition of the scaled elastic energies
{\eqref{eq:scaled-elastic-energy}}:
\[ \mathcal{I}_{\alpha}^h (y) = \frac{1}{h^{2 \alpha - 2}}  \int_{\Omega_1}
   W_0 (x_3, \nabla_h y (x)  (I + h^{\alpha - 1} B^h (x_3))) \mathd x. \]
\begin{lemma}
  \label{lem:approx-rotations}Let $\alpha \in (2, \infty)$ and let $(y^h)_{h >
  0} \subset Y$ have finite scaled $\mathcal{I}^h_{\alpha}$ energy. For every
  $h > 0$ there exist constants $\overline{R}^h \in \tmop{SO} (3)$ and $c^h
  \in \mathbb{R}^3$ such for the corrected deformations
  \begin{equation}
    \label{eq:corrected-deformations} \tilde{y}^h = \rho (y^h) \assign
    (\overline{R}^h)^{\top} y^h - c^h .
  \end{equation}
  there exist rotations $R^h : \omega \rightarrow \tmop{SO} (3)$ (extended
  constantly along $x_3$ to all of $\Omega_1$ outside $\{ 0 \} \times \omega$)
  approximating $\nabla_h  \tilde{y}^h$ in $L^2 (\Omega_1)$. Quantitatively:
  \[ \| \nabla_h  \tilde{y}^h - R^h \|_{0, 2, \Omega_1} \leqslant Ch^{\alpha -
     1} . \]
  Furthermore,
  \[ \| R^h - I \|_{0, 2, \Omega_1} \leqslant Ch^{\alpha - 2} . \]
  Finally there exists a subsequence (not relabelled) such that for the scaled
  and averaged in-plane and out-of-plane displacements from
  {\eqref{def:scaled-in-out-of-plane-displacements}} there exist $(u, v) \in
  W^{1, 2} (\omega ; \mathbb{R}^2) \times W^{2, 2} (\omega)$ such that, if
  $\alpha \neq 3$:
  \[ u_{\alpha}^h \rightharpoonup u \text{ in } W^{1, 2} (\omega ;
     \mathbb{R}^2) \text{\quad and\quad} v_{\alpha}^h \rightarrow v \text{ in
     } W^{1, 2} (\omega), \]
  If $\alpha = 3$ an analogous result holds with $u_{\theta}^h$ and
  $v_{\theta}^h$ from {\eqref{def:scaled-in-out-of-plane-displacements-theta}}.
\end{lemma}

In particular, in the sense of Definition \ref{def:ph-maps} we have that 
$(y^h)_{h > 0}$ $P^h$-converges to $v \in X_{\alpha}$ (if $\alpha \in (2,3)$), 
respectively $(u, v) \in X_{\alpha}$ (if $\alpha \ge 3$). 
\smallskip 

\begin{proof}
  This is exactly a particular case of {\cite[Lemma
  1]{friesecke_hierarchy_2006}}, estimates (84) and (85) and estimates (86)
  and (87), once we prove that if $(y^h)_{h > 0}$ have finite scaled
  $\mathcal{I}^h_{\alpha}$ energy, then they have finite scaled energy in the
  sense of {\cite{friesecke_hierarchy_2006}}.
  
  Note first that among all choices we can make for the energy density $W$
  which fulfil the assumptions in {\cite{friesecke_hierarchy_2006}}, we can
  pick $\tmop{dist}^2 (\cdot, \tmop{SO} (3))$. Therefore we will bound this
  quantity. Write $d (F) \assign \tmop{dist} (F, \tmop{SO} (3))$. We begin by
  using Assumption
  \ref{main-assumptions}.\ref{property:energy-well-growth-below}:
  \begin{eqnarray*}
    Ch^{2 \alpha - 2} & \geqslant & \int_{\Omega_1} W_0 (x_3, \nabla_h y (x) 
    (I + h^{\alpha - 1} B^h (x_3)))\\
    & \gtrsim & \int_{\Omega_1} d^2 (\nabla_h y (x)  (I + h^{\alpha - 1} B^h
    (x_3))) .
  \end{eqnarray*}
  Consider now the following:
  \begin{eqnarray*}
    d^2 (F (I + h^{\alpha - 1} B^h)) & \geqslant & \tfrac{1}{2} d^2 (F) - |
    Fh^{\alpha - 1} B^h |^2\\
    & \geqslant & \tfrac{1}{2} d^2 (F) - Ch^{2 \alpha - 2}  | 1 + d^2 (F) |\\
    & \geqslant & \tfrac{1}{4} d^2 (F) - Ch^{2 \alpha - 2} .
  \end{eqnarray*}
  But then we are done since:
  \[ h^{2 \alpha - 2} \gtrsim \int_{\Omega_1} \tfrac{1}{4} d^2 (\nabla_h y) .
  \]
\end{proof}

\begin{lemma}
  \label{lem:identification-limiting-strain}\footnote{This is almost 
  {\em word for word} {\cite[Lemma 2]{friesecke_hierarchy_2006}} with
  the very minor addition of the factors $\theta, \sqrt{\theta}$. For other
  scaling choices see {\cite[p.208]{friesecke_hierarchy_2006}}. Note that
  this is inspired by {\cite[Theorem 5.4.2]{ciarlet_mathematical_1997}}
  (itself based in {\cite[Theorem 1.4.1.c]{ciarlet_mathematical_1997}}).}Let
  $\alpha \in (2, \infty)$ and let $(y^h)_{h > 0}$ be a sequence in $Y$ which
  $P^h$-converges to $(u, v) \in X_{\alpha}$ in the sense of Theorem
  \ref{thm:gamma-hierarchy} and $R^h : \omega \rightarrow \tmop{SO} (3)$
  (extended constantly along $x_3$ to all of $\Omega_1$ outside $\{ 0 \}
  \times \omega$) such that
  
  \begin{equation*}
%      \label{eq:ident-strain-cond-Rh} 
      \| \nabla_h y^h - R^h \|_{0, 2, \Omega_1} \leqslant Ch^{\alpha - 1} .
  \end{equation*}
  
  Then:
  \begin{equation*}
%    \label{eq:ident-strain-conv-Ah} 
      A^h \assign \frac{1}{h^{\alpha - 2}}  (R^h
      - I) \longrightarrow \left\{\begin{array}{lll}
      \sqrt{\theta} A & \text{if} & \alpha = 3,\\
      A & \text{else}, & 
    \end{array}\right. \text{ in } L^2 (\omega ; \mathbb{R}^{3 \times 3}),
  \end{equation*}
  where
  \[ A = e_3 \otimes \hat{\nabla} v - \hat{\nabla} v \otimes e_3, \]
  and
  \begin{equation*}
    G^h \assign \frac{(R^h)^{\top} \nabla_h y^h - I}{h^{\alpha - 1}}
    \rightharpoonup G \text{ in } L^2 (\Omega_1 ; \mathbb{R}^{3 \times 3}),
  \end{equation*}
  where the submatrix $\check{G} \in \mathbb{R}^{2 \times 2}$ is affine in
  $x_3$:
  \begin{equation*}
%    \label{eq:representation for G} 
    \check{G} (x', x_3) = G_0 (x') + x_3 G_1(x')
  \end{equation*}
  and
  \begin{equation}
    \label{def:G1} G_1 = \left\{\begin{array}{rll}
      - \sqrt{\theta} \nabla^2 v & \text{if} & \alpha = 3,\\
      - \nabla^2 v & \text{else,} & 
    \end{array}\right.
  \end{equation}
  \begin{equation}
    \label{def:G0} \tmop{sym} G_0 = \left\{\begin{array}{rll}
      \theta \left( \grs u + \tfrac{1}{2} \nabla v \otimes \nabla v \right) &
      \text{if} & \alpha = 3,\\
      \grs u & \text{if} & \alpha > 3,
    \end{array}\right.
  \end{equation}
  and
  \[ \grs u + \tfrac{1}{2} \nabla v \otimes \nabla v = 0 \text{, if } \alpha
     \in (2, 3) . \]
\end{lemma}

\begin{proof}
  See {\cite[p.\ 208--209]{friesecke_hierarchy_2006}}.
\end{proof}

%-----------------------------------------------------------------------------------------
%-----------------------------------------------------------------------------------------
%-----------------------------------------------------------------------------------------
\section{$\mathbf{\Gamma}$-convergence of the
hierarchy}\label{sec:gamma-convergence-hierarchy}
%-----------------------------------------------------------------------------------------

This section proves the lower (Theorem \ref{thm:lower-bound}) and upper
bounds (Theorems \ref{thm:upper-bound-lki}, \ref{thm:upper-bound-vk} and
\ref{thm:upper-bound-lvk}) required for deriving the hierarchy of models in
Theorem \ref{thm:gamma-hierarchy}.

An important result of {\cite{friesecke_hierarchy_2006}} is that for small $h
> 0$ deformations $y^h$ of finite scaled energy are, up to rigid motions,
roughly the trivial map $(x', x_3) \mapsto (x', hx_3)$. The factor by which
they fail to (almost) be the identity is essential for the linearisation step
in the proof below as well as for the identification of the limit strains of
weakly convergent sequences of scaled displacements. We must account for these
rigid motions if compactness is to be achieved, in particular because
deformations might ``wander to infinity'' without altering the elastic energy.
Lemmas \ref{lem:approx-rotations} and \ref{lem:identification-limiting-strain}
gather these ideas more precisely. In particular, the last statement of
Lemma \ref{lem:approx-rotations} provides the required compactness.

Recall that we are always using weak convergence in the spaces $X_{\alpha}$.

\begin{theorem}[Lower bounds]
  \label{thm:lower-bound}Let $\alpha \in (2, 3)$. If $(y^h)_{h > 0} \subset Y$
  is a sequence $P^h_{\alpha}$-converging to $v \in X_{\alpha}$, then
  \[ \underset{h \rightarrow 0}{\tmop{linf}} \mathcal{I}_{\alpha}^h (y^h)
     \geqslant \mathcal{I}_{\rm lKi} (v)^{} . \]
  Now let $\alpha = 3$. If $(y^h)_{h > 0} \subset Y$ is a sequence
  $P^h_{\theta}$-converging to $(u, v) \in X_{\alpha}$, then for all $\theta >
  0$
  \[ \underset{h \rightarrow 0}{\tmop{linf}}  \frac{1}{\theta}
     \mathcal{I}_{\alpha}^h (y^h) \geqslant \mathcal{I}^{\theta}_{\rm vK} 
  (u, v)^{} . \]
  Finally, let $\alpha > 3$. If $(y^h)_{h > 0} \subset Y$ is a sequence
  $P^h_{\alpha}$-converging to $(u, v) \in X_{\alpha}$, then
  \[ \underset{h \rightarrow 0}{\tmop{linf}} \mathcal{I}_{\alpha}^h (y^h)
     \geqslant \mathcal{I}_{\rm lvK} (u, v)^{} . \]
\end{theorem}

\begin{proof}
  If $\alpha = 3$, we define $\tilde{B}^h \assign \sqrt{\theta} B^h$ and
  $\tilde{B} = \sqrt{\theta} B$, otherwise $\tilde{B} \assign B$ and
  $\tilde{B}^h \assign B^h$. Following closely the techniques in
  {\cite{friesecke_theorem_2002,friesecke_hierarchy_2006,schmidt_minimal_2007,schmidt_plate_2007}}
  we use a Taylor expansion of the energy around the identity which allows us
  to cancel or identify its lower order terms. For this we must correct the
  deformations with an approximation by rotations and work in adequate sets
  where there is control over higher order terms.
  
  Upon passing to a subsequence (not relabelled) which realises 
  ${\tmop{linf}}_{h \rightarrow 0} \mathcal{I}_{\alpha}^h (y^h)$ as its limit,  
  we may w.l.o.g.\ assume that $(y^h)_{h > 0}$ has finite scaled $\mathcal{I}^h_{\alpha}$
  energy and pass to further subsequences in the following. 
  \smallskip 

  {\step{1: Approximation by rotations}{We will be working with the corrected
  deformations
  \[ \rho (y^h) \assign (\overline{R}^h)^{\top} y^h - c^h, \]
  as given in Lemma \ref{lem:approx-rotations}. For simplicity we use the same
  notation $y^h$ for these functions. Also by Lemma \ref{lem:approx-rotations}
  there exist rotations $R^h : \omega \rightarrow \tmop{SO} (3)$ (extended
  constantly along $x_3$ to all of $\Omega_1$ outside $\omega \times \{ 0 \}$)
  which approximate $\nabla_h y^h$ in $L^2 (\Omega_1)$ and are close to the
  identity, as required for the identification of the limit strain in Lemma
  \ref{lem:identification-limiting-strain}.}}
  \smallskip 
 
  {\step{2: Rewriting of the deformation gradient}{The functions
  \[ G^h \assign \frac{(R^h)^{\top} \nabla_h y^h - I}{h^{\alpha - 1}} \]
  are uniformly bounded in $L^2$ by invariance of the norm by rotations:
  \begin{eqnarray}
    \| G^h \|_{0, 2, \Omega_1} 
    & = & h^{1 - \alpha}  \| \nabla_h y^h - R^h \|_{0, 2, \Omega_1} \leqslant
    C.  \label{eq:Gh-bounded}
  \end{eqnarray}
  Now, by the frame invariance of $W^h (x_3, \cdot)$
  \begin{eqnarray}
    W^h (x_3, \nabla_h y^h) & = & W^h (x_3, (R^h)^{\top} \nabla_h y^h)
    \nonumber\\
    & = & W_0 (x_3, (R^h)^{\top} \nabla_h y^h  (I + h^{\alpha - 1} 
    \tilde{B}^h (x_3))) \nonumber\\
    & = & W_0 (x_3, I + h^{\alpha - 1} A^h),  \nonumber % \label{eq:Ah-introduced}
  \end{eqnarray}
  where we have set
  \begin{eqnarray*}
    A^h (x) & \assign & \frac{(R^h)^{\top} \nabla_h y^h (x) - I}{h^{\alpha -
    1}} + (R^h)^{\top} \nabla_h y^h (x)  \tilde{B}^h (x_3)\\
    & = & G^h + (R^h)^{\top} \nabla_h y^h  \tilde{B}^h .
  \end{eqnarray*}}}
  
  {\step{3: Cutoff function}{We will be expanding $W_0 (x_3, I + h^{\alpha - 1}
  A^h)$ around $I$, but in order to apply the Taylor expansion successfully we
  need to stay where $W_0$ is twice differentiable, that is we must control
  $\tmop{dist} (I + h^{\alpha - 1} A^h, \tmop{SO} (3))$. We achieve this by
  multiplying with a cutoff function ${\large \chi}^h$, defined as the
  characteristic function of the ``good set'' $\{ x \in \Omega_1 : | G^h
  | \leqslant h^{- 1 / 2} \}$. Here we have:
  \[ h^{1 / 2} \gg h^{\alpha - 3 / 2} \geqslant {\large \chi}^h  | h^{\alpha -
     1} G^h | = {\large \chi}^h  | (R^h)^{\top} \nabla_h y^h - I | = {\large
     \chi}^h  | \nabla_h y^h - R^h |, \]
  which, because $| R^h | \equiv \sqrt{3}$, implies that ${\large \chi}^h  |
  \nabla_h y^h | \leqslant C.$ Consequently, since the $\tilde{B}^h$ are
  uniformly bounded as well:
  \begin{eqnarray}
    {\large \chi}^h  | h^{\alpha - 1} A^h | & = & {\large \chi}^h  | h^{\alpha
    - 1} G^h + h^{\alpha - 1}  (R^h)^{\top} \nabla_h y^h  \tilde{B}^h |
    \nonumber\\
    & \leqslant & {\large \chi}^h  | h^{\alpha - 1} G^h | +\mathcal{O}
    (h^{\alpha - 1}) \nonumber\\
    & = & o ( h^{1 / 2} ), \nonumber % \label{eq:est-hA}
  \end{eqnarray}
  and then
  \[ \tmop{dist} \left( I + h^{\alpha - 1}  {\large \chi}^h A^h, \tmop{SO} (3)
     \right) \leqslant | I + h^{\alpha - 1}  {\large \chi}^h A^h - I | 
     = o ( h^{1 / 2} ), \]
  so in the good sets we may indeed expand around $I$ for small values of $h$.
  Now, the sequence $(G^h)_{h > 0}$ is bounded in $L^2$ by
  {\eqref{eq:Gh-bounded}} so we may extract a subsequence converging weakly in
  $L^2$ to some $G \in L^2 (\Omega_1)$, which we consider from now on without
  relabelling. Furthermore the sequence $( {\large
  \chi}^h )_{h > 0}$ is essentially bounded and ${\large \chi}^h
  \rightarrow 1$ in measure in $\Omega_1$. Indeed $| \{ | {\large
  \chi}^h - 1 | > \varepsilon \} | = | \{ | G^h | > h^{- 1 /
  2} \} | \rightarrow 0 \text{ as } h \rightarrow 0$ because $\| G^h \|_{0, 2,
  \Omega_1} \leqslant C$ uniformly. Consequently we have
  \[ {\large \chi}^h G^h \rightharpoonup G \text{ in } L^2 (\Omega_1) . \]
  Analogously, the sequence $( {\large \chi}^h  \tilde{B}^h )_{h >
  0}$ is essentially bounded and converges in measure to $\tilde{B}$ because
  $| \{ | {\large \chi}^h  \tilde{B}^h - \tilde{B} | >
  \varepsilon \} | \leqslant | \{ | \tilde{B}^h - \tilde{B} | >
  \varepsilon \} | + | \{ {\large \chi}^h = 0 \} \cap \{ |
  \tilde{B}^{} | > \varepsilon \} | \rightarrow 0$. Hence, using again
  the strong convergence $(R^h)^{\top} \nabla_h y^h \rightarrow I$ in $L^2
  (\Omega_1)$ (Lemma \ref{lem:approx-rotations}):
  \[ (R^h)^{\top} \nabla_h y^h  {\large \chi}^h  \tilde{B}^h \rightharpoonup
     \tilde{B} \text{ in } L^2 (\Omega_1) . \]
  So we conclude
  \begin{equation*}
%    \label{eq:convergence-Ah} 
    {\large \chi}^h A^h \rightharpoonup A \assign G
    + \tilde{B} \text{ in } L^2 (\Omega_1) .
  \end{equation*}}}
  
  {\step{4: Taylor expansion}{Because $W_0 (x_3, \cdot)_{| \tmop{SO} (3)} \equiv
  0$, for any fixed $x_3$ the lower order terms of its Taylor expansion
  \[ W_0 (x_3, I + E) = W_0 (x_3, I) + D W_0 (x_3, I) [E] + \frac{1}{2} D^2
     W_0 (x_3, I) [E, E] + o (| E |^2) \]
  vanish and we have (for small enough $h$, as explained above)
  \[ W_0 \left( x_3, I + h^{\alpha - 1}  {\large \chi}^h A^h \right) =
     \frac{1}{2} Q_3 \left( x_3, h^{\alpha - 1}  {\large \chi}^h A^h \right) +
     \eta^h \left( x_3, h^{\alpha - 1}  {\large \chi}^h A^h \right), \]
  where $\eta^h ( x_3, h^{\alpha - 1}  {\large \chi}^h A^h ) = o
  ( h^{2 \alpha - 2} | {\large \chi}^h A^h |^2 )$
  represents the higher order terms. Defining the uniform bound
  \[ \omega (s) \assign \underset{- 1 \leqslant 2 r \leqslant 1}{\tmop{ess}
     \sup}  \underset{| M | \leqslant s}{\sup}  | \eta^h (r, M) |, \]
  we have $\omega (s) = o (s^2)$ by Assumption
  \ref{main-assumptions}.\ref{assumption:modulus}, and integrating over the
  rescaled domain $\Omega_1$ we obtain the estimate:
  
\begin{align}
   \label{eq:first-est}
   & \frac{1}{h^{2 \alpha - 2}} \int_{\Omega_1} W^h (x_3,
   \nabla_h y^h) \mathd x \notag \\ 
   & \quad \geqslant~ \frac{1}{h^{2 \alpha - 2}}  \int_{\Omega_1}
   W^h ( x_3, I + {\large \chi}^h h^{\alpha - 1} A^h ) \mathd x \notag \\ 
   & \quad \geqslant~ \frac{1}{h^{2 \alpha - 2}}  \int_{\Omega_1} \frac{h^{2 \alpha -
   2}}{2} Q_3 ( x_3, {\large \chi}^h A^h ) - \omega ( |
   h^{\alpha - 1}  {\large \chi}^h A^h | ) \mathd x \notag \\ 
   & \quad =~ \frac{1}{2} 
   \int_{\Omega_1} Q_3 ( x_3, {\large \chi}^h A^h ) - \frac{1}{h^{2
   \alpha - 2}}  \int_{\Omega_1} \omega ( | h^{\alpha - 1}  {\large
   \chi}^h A^h | ) \mathd x. 
\end{align}
  }}
  \smallskip
  
  {\step{5: The limit inferior}{In order to pass to the limit, for the first
  integral on the right hand side of {\eqref{eq:first-est}} we use that $Q_3$
  is positive semidefinite, therefore convex and continuous, and the integral
  is sequentially weakly lower semicontinuous. For the second integral we use
  again Assumption \ref{main-assumptions}.\ref{assumption:modulus} and the
  fact that $| h^{\alpha - 1}  {\large \chi}^h A^h | \leqslant h^{1
  / 2}$ to obtain the bound (uniform over $\Omega_1$):
  \[ \frac{\omega ( | h^{\alpha - 1}  {\large \chi}^h A^h |
     )}{| h^{\alpha - 1}  {\large \chi}^h A^h |^2} \leqslant
     \underset{| s | \leqslant h^{1 / 2}}{\sup}  \frac{\omega (s)}{s^2}
     \longrightarrow 0 \text{ as } h \rightarrow 0. \]
  But then, because ${\large \chi}^h A^h$ converges weakly in $L^2$, we have
  $\| {\large \chi}^h A^h \|_{0, 2, \Omega_1}^2 \leqslant C$ and
  \begin{eqnarray*}
    \frac{1}{h^{2 \alpha - 2}}  \int_{\Omega_1} \omega \left( \left| h^{\alpha
    - 1}  {\large \chi}^h A^h \right| \right) \mathd x & = & \int_{\Omega_1}
    \frac{\omega \left( \left| h^{\alpha - 1}  {\large \chi}^h A^h \right|
    \right)}{\left| h^{\alpha - 1}  {\large \chi}^h A^h \right|^2} 
    \frac{\left| h^{\alpha - 1}  {\large \chi}^h A^h \right|^2}{h^{2 \alpha -
    2}} \mathd x\\
    & \leqslant & \underset{| s | \leqslant h^{1 / 2}}{\sup}  \frac{\omega
    (s)}{s^2} \underbrace{\int_{\Omega_1} \left| {\large \chi}^h A^h \right|^2
    \mathd x}_{\text{uniformly bded.}} \longrightarrow 0
  \end{eqnarray*}
  \text{ as }$h \rightarrow 0$. Taking the $\lim \inf$ at both sides of
  {\eqref{eq:first-est}} we have:
  \begin{eqnarray*}
    \applicationspace{5 \tmop{em}} \underset{h \rightarrow 0}{\tmop{linf}} 
    \frac{1}{h^{2 \alpha - 2}}  \int_{\Omega_1} W^h (x_3, \nabla_h y^h) \mathd
    x %&  & \\
    & \geqslant & \underset{h \rightarrow 0}{\tmop{linf}}  \frac{1}{2} 
    \int_{\Omega_1} Q_3 \left( x_3, {\large \chi}^h A^h \right) \mathd x\\
    &  & \applicationspace{2 \tmop{em}} - \underset{h \rightarrow 0}{\lim} 
    \frac{1}{h^{2 \alpha - 2}}  \int_{\Omega_1} \omega (| h^{\alpha - 1} A^h
    |) \mathd x\\
    & \geqslant & \frac{1}{2}  \int_{\Omega_1} Q_3 (x_3, G + \tilde{B})
    \mathd x\\
    & \geqslant & \frac{1}{2}  \int_{\Omega_1} Q_2 (x_3, \check{G} +
    \check{\tilde{B}}) \mathd x,
  \end{eqnarray*}
  where the last estimate follows trivially from the definition of $Q_2$.
  \smallskip 

  \noindent If $\alpha \geqslant 3$, by Lemma
  \ref{lem:identification-limiting-strain} the limit strain $\check{G}$ has
  the representation
  \[ \check{G} (x) = G_0 (x') + x_3 G_1 (x'), \]
  with $G_1$ and $\tmop{sym} G_0$ given respectively by {\eqref{def:G1}} and
  {\eqref{def:G0}} as:
  \[ G_1 = \left\{\begin{array}{rll}
       - \sqrt{\theta} \nabla^2 v & \text{if} & \alpha = 3,\\
       - \nabla^2 v & \text{if} & \alpha > 3,
     \end{array}\right. \]
  and
  \[ \tmop{sym} G_0 = \left\{\begin{array}{rll}
       \theta \left( \grs u + \tfrac{1}{2} \nabla v \otimes \nabla v \right) &
       \text{if} & \alpha = 3,\\
       \grs u & \text{if} & \alpha > 3.
     \end{array}\right. \]
  We plug both into the last integral and use the fact that $Q_2 (x_3, \cdot)$
  vanishes on antisymmetric matrices to obtain
  \begin{align*}
    &\applicationspace{5 \tmop{em}} \underset{h \rightarrow 0}{\tmop{linf}} 
    \frac{1}{h^{2 \alpha - 2}}  \int_{\Omega_1} W^h_{\alpha} (x_3, \nabla_h
    y^h) \mathd x \\ 
    &\qquad \geqslant~ \frac{1}{2}  \int_{\Omega_1} Q_2 (x_3, G_0 (x') \nobracket
    + x_3 G_1 (x') + \check{\tilde{B}} (x_3)) \mathd x \applicationspace{1
    \tmop{em}}\\
    &\qquad =~ \frac{1}{2}  \int_{\omega} \overline{Q}_2 (\tmop{sym} G_0, G_1)
    \mathd x' .
  \end{align*}
  In particular, if $\alpha = 3$, we have again:
  \begin{eqnarray*}
    \underset{h \rightarrow 0}{\tmop{linf}}  \frac{1}{\theta h^4} 
    \int_{\Omega_1} W^h_{\alpha} (x_3, \nabla_h y^h) \mathd x & \geqslant &
    \frac{1}{2 \theta}  \int_{\omega} \overline{Q}_2 (\tmop{sym} G_0, G_1)
    \mathd x'\\
    & = & \frac{1}{2}  \int_{\omega} \overline{Q}_2 (\theta^{1 / 2} (\grs u +
    \tfrac{1}{2} \nabla v \otimes \nabla v), - \nabla^2 v) .
  \end{eqnarray*}

  {\noindent}If $\alpha \in (2, 3)$, then $\tmop{sym} G_0$ is
  unknown, so we must further relax the integrand. With the definition of
  $\overline{Q}_2^{\star}$ we see that the final integral above is
  \[ \frac{1}{2}  \int_{\Omega_1} Q_2 (x_3, G_0 - x_3 \nabla^2 v + \check{B})
     \mathd x \geqslant \frac{1}{2}  \int_{\omega} \overline{Q}^{\star}_2 (-
     \nabla^2 v) \mathd x' . \]}}
\end{proof}

We proceed now with the computation of the recovery sequences for each of the
three regimes discussed. We assume convexity of the domain in order to apply
the representation theorems in Section
\ref{sec:approximation-and-representation}.

\begin{theorem}[Upper bound, linearised Kirchhoff regime]
  \label{thm:upper-bound-lki}Assume $\omega$ is convex, let $\alpha \in (2,
  3)$ and $v \in X_{\alpha} \assign W^{1, 2} (\omega)$. There exists a
  sequence $(y^h)_{h > 0} \subset Y$ which $P^h$-converges to $v$ such that
  \[ \underset{h \rightarrow 0}{\tmop{lsup}} \mathcal{I}_{\alpha}^h (y^h)
     \leqslant \mathcal{I}_{\rm lKi} (v)^{}, \]
  with $\mathcal{I}_{\rm lKi}$ defined as in
  {\eqref{eq:energy-lki}} by
  \[ \mathcal{I}_{\rm lKi} (v) \assign
     \left\{\begin{array}{rl}
       \frac{1}{2}  \int_{\omega} \overline{Q}_2^{\star} (\nabla^2 v (x'))
       \mathd x' & \text{ if } v \in W^{2, 2}_{s h} (\omega),\\
       \infty & \text{ otherwise} .
     \end{array}\right. \]
\end{theorem}

\begin{proof}
  We set $\varepsilon = h^{\alpha - 2}$, so that $h \ll \varepsilon \ll 1$ and
  $h^2 \ll \varepsilon h \ll 1$.
  \smallskip 

  {\step{1: Setup and recovery sequence}{The functional $\mathcal{I}_{\rm lKi}$ 
  is strongly continuous on $W^{2, 2}_{s h} (\omega)$ by the
  continuity and 2-growth of $\overline{Q}^{\star}_2$. By Theorem
  \ref{thm:density-singular-hessian} we have a set $\mathcal{V}_0$ of smooth
  maps with singular Hessian which is $W^{2, 2}$-dense in $W^{2, 2}_{s h}$,
  see {\eqref{def:V0}}. Therefore, by a standard argument (see, e.g., 
  \cite{braides_handbook_2006}) 
   it is enough to construct here the recovery
  sequence. Take then a smooth function $v \in \mathcal{V}_0$. Because $\|
  \nabla v \|_{\infty} < C$, for $\varepsilon$ small enough there exist by
  {\cite[Theorem 7]{friesecke_hierarchy_2006}} in-plane displacements
  $u_{\varepsilon} \in W^{2, 2} (\omega ; \mathbb{R}^2) \cap W^{2, \infty}
  (\omega ; \mathbb{R}^2)$ with uniform bounds in $\varepsilon$ such that the
  deformations
  \[ \overline{y}_{\varepsilon} (x') \assign \left(\begin{array}{c}
       x' + \varepsilon^2 u_{\varepsilon} (x')\\
       \varepsilon v (x')
     \end{array}\right) \]
  are isometries.\footnote{The uniform bounds for $\| u_{\varepsilon} \|_{2,
  2}$ follow from {\cite[Theorem 7]{friesecke_hierarchy_2006}}, equation
  (181), and those for $\| u_{\varepsilon} \|_{2, \infty}$ from the explicit
  construction done in the proof, in particular equations (183), (186) and
  (190).} That is: $\nabla^{\top} \overline{y}_{\varepsilon} \nabla
  \overline{y}_{\varepsilon} = I_2$, where
  \[ \nabla \overline{y}_{\varepsilon} = \left(\begin{array}{c}
       I_2 \\
       0 ~~ 0
     \end{array}\right) + \varepsilon \left(\begin{array}{cc}
       0_2 & \\
       \nabla^{\top} v & 
     \end{array}\right) + \varepsilon^2  \left(\begin{array}{c}
       \nabla u_{\varepsilon} \\
       0 ~~ 0
     \end{array}\right) \in \mathbb{R}^{3 \times 2} . \]
  Additionally the following normal vectors are unitary in $\mathbb{R}^3$:
  \begin{eqnarray*}
    b_{\varepsilon} (x') & \assign & \overline{y}_{\varepsilon, 1} (x') \wedge
    \overline{y}_{\varepsilon, 2} (x')\\
    & = & - \varepsilon \left(\begin{array}{c}
      \nabla v\\
      0
    \end{array}\right) + \left(\begin{array}{c}
      \varepsilon^3 \nabla u_{\varepsilon \nocomma 2} \cdot (v_{, 2}, - v_{,
      1})\\
      \varepsilon^3 \nabla u_{\varepsilon \nocomma 1} \cdot (- v_{, 2}, v_{,
      1})\\
      1 + \varepsilon^2 \tmop{tr} \nabla u_{\varepsilon} + \varepsilon^4 \det
      \nabla u_{\varepsilon}
    \end{array}\right)\\
    & = & e_3 - \varepsilon \hat{\nabla} v (x') + r_{\varepsilon} (x'),
  \end{eqnarray*}
  where the rest $r_{\varepsilon}$ satisfies
  \[ \| r_{\varepsilon} \|_{1, \infty} =\mathcal{O} (\varepsilon^2) \]
  by virtue of $\| u_{\varepsilon} \|_{2, \infty} \leqslant C$ and $\| \nabla
  v \|_{\infty} \leqslant C$. Consequently the matrices
  \[ R_{\varepsilon} \assign (\nabla \overline{y}_{\varepsilon},
     b_{\varepsilon}) = I + \varepsilon \left(\begin{array}{cc}
       0 & - \nabla v\\
       \nabla^{\top} v & 0
     \end{array}\right) + \underbrace{r_{\varepsilon} \otimes e_3 +
     \varepsilon^2  \hat{\nabla} u_{\varepsilon}}_{\backassign
     \tilde{r}_{\varepsilon}} \]
  are in $\tmop{SO} (3)$ for every $x' \in \omega$, with the remaining matrix
  $\tilde{r}_{\varepsilon}$ satisfying
  \[ \| \tilde{r}_{\varepsilon} \|_{1, \infty} =\mathcal{O} (\varepsilon^2) \]
  by the same arguments as before. Now, for some smooth functions $\alpha,
  g_1, g_2 \in C^{\infty} (\overline{\omega} ; \mathbb{R})$, $g \assign (g_1,
  g_2)$ and $d \in L^{\infty} (\Omega_1 ; \mathbb{R}^3)$ with $\nabla' d \in
  L^{\infty} (\Omega_1 ; \mathbb{R}^{3 \times 2})$ and $D^h \in C^{\infty}
  (\overline{\Omega}_1 ; \mathbb{R}^3)$ to be determined later, set
  \begin{eqnarray}
    y^h (x', x_3) & \assign & \overline{y}_{\varepsilon} (x') + h (x_3 -
    \alpha (x')) b_{\varepsilon} (x') + \varepsilon h (g (x'), 0) \nonumber\\
    &  & \applicationspace{1 \tmop{em}} + \varepsilon h^2  \int_0^{x_3} d
    (x', \xi) \mathd \xi + D^h (x', x_3) .  \label{eq:thm:lki:recovery}
  \end{eqnarray}
  We will prove
  \[ \mathcal{I}_{\alpha}^h (y^h)  \underset{h \rightarrow 0}{\longrightarrow}
     \mathcal{I}_{\rm lKi} (v) . \]
  as well as $P^h_{\alpha}  (y^h) \rightarrow v$ in $W^{1, 2}$ for some
  constants $R^h \in \tmop{SO} (3), c^h \in \mathbb{R}^3$.}}
  \smallskip 
  
  {\step{2: Preliminary computations}{In order to compute the limit of
  $\frac{1}{h^{2 \alpha - 2}}  \int_{\Omega_1} W_0 (x_3, \nabla_h y^h  (I +
  \varepsilon hB^h))$ we start with the gradient of the recovery sequence:
  \begin{eqnarray*}
    \nabla_h y^h & = & (\nabla \overline{y}_{\varepsilon}, 0) + h \nabla_h 
    [(x_3 - \alpha) b_{\varepsilon}]\\
    &  & \applicationspace{1 \tmop{em}} + \varepsilon h [\hat{\nabla} g + d
    \otimes e_3] + \nabla_h D^h + o (\varepsilon h) .
  \end{eqnarray*}
  For the term in $h$ and any $i \in \{ 1, 2, 3 \}$ and $j \in \{ 1, 2 \}$ we
  have
  \[ \partial_j [(x_3 - \alpha (x')) b_{\varepsilon} (x')]_i = \partial_j
     [(x_3 - \alpha) b_{\varepsilon \nocomma i}] = (x_3 - \alpha)
     b_{\varepsilon \nocomma i, j} - \varepsilon \alpha_{, j} b_{\varepsilon
     \nocomma i} . \]
  Also: $\frac{1}{h} \partial_3 [(x_3 - \alpha) b_{\varepsilon}] = \frac{1}{h}
  b_{\varepsilon}$, so that
  \begin{eqnarray*}
    \nabla_h [(x_3 - \alpha) b_{\varepsilon}] & = & (x_3 - \alpha) 
    \hat{\nabla} b_{\varepsilon} - b_{\varepsilon} \otimes \hat{\nabla} \alpha
    + \tfrac{1}{h} b_{\varepsilon} \otimes e_3\\
    & = & (\alpha - x_3)  (\varepsilon \hat{\nabla}^2 v - \hat{\nabla}
    r_{\varepsilon}) - b_{\varepsilon} \otimes \hat{\nabla} \alpha +
    \tfrac{1}{h} b_{\varepsilon} \otimes e_3 .
  \end{eqnarray*}
  Substituting back into the gradient yields:
  \begin{eqnarray}
    \nabla_h y^h & = & R_{\varepsilon} + \varepsilon h \underbrace{[(\alpha -
    x_3)  \hat{\nabla}^2 v + \hat{\nabla} g + d \otimes e_3 + o
    (1)]}_{\backassign A^h} \nonumber\\
    &  & \applicationspace{1 \tmop{em}} - hb_{\varepsilon} \otimes
    \hat{\nabla} \alpha + \nabla_h D^h . \nonumber % \label{eq:thm:lki:grad}
  \end{eqnarray}
  Because we intend to use the frame invariance of the energy, we will need
  the product of $\nabla_h y^h$ with $R_{\varepsilon}^{\top} = I +\mathcal{O}
  (\varepsilon)$. First we have:
  \[ \varepsilon hR_{\varepsilon}^{\top} A^h = \varepsilon hA^h + o
     (\varepsilon h) = \varepsilon hA^h, \]
  where we have subsumed terms $o (\varepsilon h)$ into the $o (1)$ inside
  $A^h$. Using $| b_{\varepsilon} | \equiv 1$ and $\overline{y}_{\varepsilon,
  i} \perp b_{\varepsilon}$ we also have $R_{\varepsilon}^{\top}
  b_{\varepsilon} = e_3$. Therefore
  \begin{equation}
    \label{eq:thm:lki:grad2} R_{\varepsilon}^{\top} \nabla_h y^h = I_3 +
    \varepsilon hA^h \underbrace{- he_3 \otimes \hat{\nabla} \alpha +
    R_{\varepsilon}^{\top} \nabla_h D^h}_{= : F^h} .
  \end{equation}}}
  \smallskip 
  
  {\step{3: Convergence of the energies}{The next step is a Taylor expansion
  around the identity. Given that the energy is scaled by $(\varepsilon h)^{-
  2}$, only those terms scaling as $\varepsilon h$ in
  {\eqref{eq:thm:lki:grad2}} will remain: anything beyond that will not be
  seen and anything below will make the energy blow up. This means that we
  must choose $D^h$ so that $F^h = o (\varepsilon h)$. In
  {\cite{friesecke_hierarchy_2006}}, {\cite{schmidt_plate_2007}} it was
  possible for the authors to obtain exactly $F^h = 0$ by choosing $D^h$
  adequately, but in our case this will not be possible.\footnote{Technically,
  this is due to the fact that the term $he_3 \otimes \hat{\nabla} \alpha$ is
  a row in a matrix instead of a column, which makes it impossible to exactly
  compensate because $R_{\varepsilon}^{\top} \nabla_h D^h$ effectively only
  provides a column vector to work with. Indeed,
  \[ R_{\varepsilon}^{\top} \nabla_h D^h = \nabla_h D^h + \varepsilon
     \left(\begin{array}{cc}
       0 & \nabla v\\
       - \nabla^{\top} v & 0
     \end{array}\right) \nabla_h D^h + \tilde{r}^{\top}_{\varepsilon} \nabla_h
     D^h, \]
  {\noindent}so in order to cancel $he_3 \otimes \hat{\nabla} \alpha$ we must
  have that the leading term $\nabla_h D^h$ be of order $h$. But then
  $\nabla_h D^h = \left( \nabla' D^h, \frac{1}{h} D_{, 3}^h \right)$ requires
  that $D^h$ scale at least as $h^2 \ll \varepsilon h \ll 1$ so we ``lose''
  the first two columns of $\nabla_h D^h$.} If we set $D^h \assign h^2 D$ for
  some smooth $D$, we have
  \begin{eqnarray*}
    F^h & = & h [D_{, 3} \otimes e_3 + \varepsilon (v_{, 1} D_{3, 3}, v_{, 2}
    D_{3, 3}, - v_{, 1} D_{1, 3} - v_{, 2} D_{2, 3}) \otimes e_3 \nobracket\\
    &  & \applicationspace{1 \tmop{em}} \nobracket - e_3 \otimes \hat{\nabla}
    \alpha + o (\varepsilon)]\\
    & = : & h \tilde{F}^h .
  \end{eqnarray*}
  This means that we must solve the equations $\tilde{F}^h = o (\varepsilon)$.
  Although these have no solution the symmetrised version
  does,\footnote{Dividing by $h$ we arrive at:
  \[ \left\{\begin{array}{rll}
       D_{1, 3} + \varepsilon v_{, 1} D_{3, 3} & = & \alpha_{, 1} + o
       (\varepsilon),\\
       D_{2, 3} + \varepsilon v_{, 2} D_{3, 3} & = & \alpha_{, 2} + o
       (\varepsilon),\\
       D_{3, 3} - \varepsilon v_{, 1} D_{1, 3} - \varepsilon v_{, 2} D_{2, 3}
       & = & o (\varepsilon),
     \end{array}\right. \]
  with solution:
  \[ D (x', x_3) = x_3  \hat{\nabla} \alpha + x_3 \varepsilon \nabla v \cdot
     \nabla \alpha e_3 . \]} so that for every smooth choice of $\alpha$ we
  can pick a bounded $D^h$ such that
  \begin{equation}
    \label{eq:thm:lki:fh} \tilde{F}^h_s = 0 \text{, \ and \ } \tilde{F}^h
    =\mathcal{O} (1),
  \end{equation}
  a fact that we will exploit next. By frame invariance,
  {\eqref{eq:thm:lki:grad2}} and $F^h = h \tilde{F}^h$, we can write
  \begin{eqnarray*}
    \applicationspace{2 \tmop{em}} W_0 (x_3, \nabla_h y^h  (I + \varepsilon
    hB^h)) 
    & = & W_0 (x_3, R_{\varepsilon}^{\top} \nabla_h y^h  (I + \varepsilon
    hB^h))\\
    \applicationspace{2 \tmop{em}} & = & W_0 (x_3, (I + \varepsilon hA^h + h
    \tilde{F}^h) (I + \varepsilon hB^h))\\
    & = & W_0 (x_3, I + h \underbrace{(\varepsilon (A^h + B^h) + \tilde{F}^h
    + o (\varepsilon))}_{= : C^h}) .
  \end{eqnarray*}
  Because of {\eqref{eq:thm:lki:fh}} by our choice of $D$ we need to subtract
  the antisymmetric part of $\tilde{F}^h$, which we do by means of another
  rotation and frame invariance:
  \begin{eqnarray*}
    W_0 (x_3, I + hC^h) & = & W_0 (x_3, \mathe^{- h \tilde{F}_a^h}  (I +
    hC^h))\\
    & = & W_0 (x_3, (I - h \tilde{F}^h_a +\mathcal{O} (h^2)) (I + hC^h))\\
    & = & W_0 (x_3, I + hC^h - h \tilde{F}^h_a +\mathcal{O} (h^2))\\
    & = & W_0 (x_3, I + \varepsilon h (A^h + B^h) + o (\varepsilon h)) .
  \end{eqnarray*}
  Now whenever $h$ is small enough that $I + hC^h$ belongs to the
  neighbourhood of $\tmop{SO} (3)$ where $W_0$ is twice differentiable, we can
  apply Taylor's theorem and the fact that $Q_3$ vanishes on antisymmetric
  matrices to see that, as $h
  \rightarrow 0$:
  \begin{eqnarray*}
    \frac{1}{\varepsilon^2 h^2} W_0 (x_3, \nabla_h y^h  (I + \varepsilon
    hB^h)) & = & \frac{1}{2} Q_3 (x_3, (A^h + B^h)_s) + o (1)\\
    & \rightarrow & \frac{1}{2} Q_3 (x_3, A_s + B_s)
  \end{eqnarray*}
  where
  \[ A_s = (\alpha - x_3)  \hat{\nabla}^2 v + \hat{\nabla}_s g + (d \otimes
     e_3)_s . \]
  We choose
  \[ d (x', x_3) =\mathcal{L} (x_3, (\alpha - x_3) \nabla^2 v + \nabla_s g +
     \check{B}_s) - B_{\cdot 3}, \]
  with $\mathcal{L}$ the map from \eqref{eq:mapping-L}, which by 
  \eqref{eq:mapping-L} and \eqref{eq:Q2-L-bounds} is linear in the second 
  component and satisfies 
  $| \mathcal{L} (t,A) | \lesssim | A |$ uniformly in $t$, and 
  $B_{\cdot 3}$ the third column of $B$. Because the matrix $(\alpha - x_3)
  \nabla^2 v + \nabla_s g + \check{B}_s$ is bounded uniformly in $x'$, by
  the bound \eqref{eq:Q2-L-bounds} the map 
  \[ x \mapsto \int_0^{x_3} \mathcal{L} (\xi, (\alpha - \xi)  \hat{\nabla}^2 v
     + \hat{\nabla}_s g + B_s (\xi)) \mathd \xi \]
  is in $W^{1, \infty} (\Omega_1 ; \mathbb{R}^3)$ and $y^h \in W^{1, 2}$ as
  required (for the derivatives with respect to $x'$ note that $v, g$ are smooth and $B$
  independent of $x'$).
  
  Now, all quantities being bounded, by dominated convergence:
  \begin{eqnarray*}
    \mathcal{I}_{\alpha}^h (y^h) & \rightarrow & \frac{1}{2}  \int_{\Omega_1}
    Q_3 (x_3, (\alpha - x_3)  \hat{\nabla}^2 v + \hat{\nabla}_s g + (d \otimes
    e_3)_s + B_s)\\
    & = & \frac{1}{2}  \int_{\Omega_1} Q_2 \left( x_3, (\alpha - x_3)
    \nabla^2 v + \grs g + \check{B}_s \right) .
  \end{eqnarray*}
  Note that a final step is required to obtain convergence to 
  $\mathcal{I}_{\rm lKi} (v)$.}}
  \smallskip 
  
  {\step{4: Convergence of the deformations: $P^h_{\alpha}  (y^h) \rightarrow v$
  in $W^{1, 2}$}{Choose $R^h \equiv I \in \tmop{SO} (3), c^h \equiv 0 \in
  \mathbb{R}^3$ in the definition of $\rho$ for
  {\eqref{def:scaled-in-out-of-plane-displacements}}. We have
  \[ P^h_{\alpha}  (y^h) = \frac{1}{\varepsilon}  \int_{- 1 / 2}^{1 / 2} y^h_3
     (x', x_3) \mathd x_3, \]
  where in {\eqref{eq:thm:lki:recovery}} we defined $y^h_3 (x', x_3) =
  \varepsilon v (x') + h (x_3 - \alpha (x')) b_{\varepsilon \nocomma 3} (x')
  +\mathcal{O} (\varepsilon h)$. Then:
  \begin{eqnarray*}
    | P^h_{\alpha}  (y^h) - v |^2 & = & \bigg| \frac{1}{\varepsilon}  \int_{-
    1 / 2}^{1 / 2} [\varepsilon v + h (x_3 - \alpha) b_{\varepsilon \nocomma
    3} +\mathcal{O} (\varepsilon h)] \mathd x_3 - v \bigg|^2\\
    & = & \mathcal{O} (\varepsilon^{- 2} h^2),
  \end{eqnarray*}
  and consequently $\| P^h_{\alpha} (y^h) - v \|_{0, 2} \rightarrow 0$. An
  analogous computation for the derivatives shows strong convergence in $W^{1,
  2}$.}}
  \smallskip 
  
  {\step{5: Simultaneous convergence}{Finally, as in {\cite[Theorem
  3.2]{schmidt_plate_2007}}, in order for the energy to converge to the true
  limit, we must pick $\alpha$ and $g$ in {\eqref{eq:thm:lki:recovery}} so as
  to approximate the minimum $\overline{Q}_2$. This is done with Corollary
  \ref{cor:representation-matrix-minimizer}, substituting sequences of smooth
  functions $(\alpha_k)_{k \in \mathbb{N}}, (g_k)_{k \in \mathbb{N}}$ for the
  functions $\alpha, g$. Then, for each fixed $k$ we have:
  \begin{eqnarray*}
    \mathcal{I}_{\alpha}^h (y_k^h) & \underset{h \rightarrow 0}{\rightarrow} &
    \frac{1}{2}  \int_{\Omega_1} Q_2 \left( x_3, (\alpha_k - x_3) \nabla^2 v +
    \grs g_k + \check{B}_s \right)\\
    & = & \frac{1}{2}  \int_{\omega} \overline{Q}^{\star}_2 (- \nabla^2 v)
    \mathd x' + o (1)_{k \rightarrow \infty},
  \end{eqnarray*}
  and
  \[ \| P^h_{\alpha}  (y^h_k) - v \|_{1, 2}^2 \leqslant C (k) \varepsilon^{-
     2} h^2 . \]
  And by a diagonal argument we can find $(y^h)_{h > 0}$ whose energy
  converges to $\mathcal{I}_{\rm lKi} (v)$ while maintaining
  the convergence of the deformations.}}
\end{proof}

\begin{theorem}[Upper bound, von Kármán regime]
  \label{thm:upper-bound-vk}Let $\alpha = 3$ and consider displacements $(u,
  v) \in X_{\alpha = 3} \assign W^{1, 2} (\omega ; \mathbb{R}^2) \times W^{1,
  2} (\omega ; \mathbb{R})$. There exists a sequence $(y^h)_{h > 0} \subset Y$
  which $P^h_{\theta}$-converges to $(u, v)$ such that
  \[ \underset{h \rightarrow 0}{\lim}  \frac{1}{\theta} \mathcal{I}_{\alpha}^h
     (y^h) =\mathcal{I}^{\theta}_{\rm vK} (u, v)^{}, \]
  with $\mathcal{I}^{\theta}_{\rm vK}$ defined as in
  {\eqref{eq:energy-vk}} by
  \[ \mathcal{I}^{\theta}_{\rm vK} (u, v) \assign \frac{1}{2} 
     \int_{\omega} \overline{Q}_2 (\theta^{1 / 2} (\grs u + \tfrac{1}{2}
     \nabla v \otimes \nabla v), - \nabla^2 v) \]
  over $X_{\alpha = 3}^0 = W^{1, 2} (\omega ; \mathbb{R}^2) \times W^{2, 2}
  (\omega ; \mathbb{R})$ and as $\infty$ elsewhere.
\end{theorem}

\begin{proof}
  In order to build the recovery sequence $(y^h)_{h > 0}$ we will use the map
  $\mathcal{L}: \left( - 1 / 2, 1 / 2 \right) \times \mathbb{R}^{2 \times 2}
  \rightarrow \mathbb{R}^3$ given by \eqref{eq:mapping-L}, which for
  each $t$ realises the minimum of $Q_3 (t, \hat{A} + c \otimes e_3), A \in
  \mathbb{R}^{2 \times 2}$, i.e.
  \[ Q_2 (t, A) = Q_3 (t, \hat{A} +\mathcal{L} (t, A) \otimes e_3) = Q_3 (t,
     \hat{A} + (\mathcal{L} (t, A) \otimes e_3)_s), \]
  where the last equality follows from the fact that $Q_2$ vanishes on
  antisymmetric matrices. Recall from \eqref{eq:Q2-L-bounds} that 
  $\mathcal{L} (t, \cdot)$ is linear for every $t$ and that $| \mathcal{L} (t,
  A) | \lesssim | A |$ uniformly in $t$.
  
  The functional $\mathcal{I}^{\theta}_{\rm vK}$ is clearly continuous
  in $X^0_{\alpha} = W^{1, 2} (\omega ; \mathbb{R}^2) \times W^{2, 2} (\omega
  ; \mathbb{R})$ with the {\em strong} topologies, so a standard argument 
  {\cite{braides_handbook_2006}} shows that it is enough to consider $(u, v) \in
  C^{\infty} (\overline{\omega} ; \mathbb{R}^2) \times C^{\infty}
  (\overline{\omega} ; \mathbb{R})$, which is dense in $X^0_{\alpha}$. We
  define:
  \begin{eqnarray}
    y^h (x', x_3) & \assign & \left(\begin{array}{c}
      x'\\
      hx_3
    \end{array}\right) + \left(\begin{array}{c}
      \theta h^2 u (x')\\
      \sqrt{\theta} hv (x')
    \end{array}\right) - \sqrt{\theta} h^2 x_3  \left(\begin{array}{c}
      \nabla v (x')\\
      0
    \end{array}\right) \nonumber\\
    &  & \applicationspace{1 \tmop{em}} + \theta h^3 d (x', x_3) 
    \nonumber %\label{eq:vK-Ansatz}
  \end{eqnarray}
  where $d \in W^{1, \infty} (\Omega_1 ; \mathbb{R}^3)$ is a vector field to
  be determined along the proof.
  \smallskip 
  
  {\step{1: Approximation of the energy}{A direct computation yields
  \begin{eqnarray*}
    \nabla_h y^h & = & I + \left(\begin{array}{cc|c}
      \theta h^2 \nabla u &  & - h \sqrt{\theta} \nabla v\\
      \hline
 %     &  & \\
 %     \hline
      h \sqrt{\theta} \nabla^{\top} v &  & 0
    \end{array}\right) - h^2 \theta \left(\begin{array}{ccc}
      x_3 \theta^{- 1 / 2} \nabla^2 v &  & 0\\
%      &  & \\
      0 &  & 0
    \end{array}\right)\\
    &  & \quad + h^2 \theta \partial_3 d \otimes e_3 +\mathcal{O} (h^3)\\
    & = & I + h \sqrt{\theta}  \underset{E}{\underbrace{(e_3 \otimes
    \hat{\nabla} v - \hat{\nabla} v \otimes e_3)}}\\
    &  & \quad + h^2 \theta \underset{F}{\underbrace{\big( \hat{\nabla} u -
    x_3 \theta^{- 1 / 2}  \hat{\nabla}^2 v + \partial_3 d \otimes e_3
    \big)}} +\mathcal{O} (h^3) .
  \end{eqnarray*}
  For later use we note here the product:
  \begin{eqnarray*}
    \nabla_h^{\top} y^h \nabla_h y^h & = & \big( I + h \sqrt{\theta} E^{\top}
    + h^2 \theta F^{\top} \big)  \big( I + h \sqrt{\theta} E + h^2 \theta F
    \big) +\mathcal{O} (h^3)\\
    & = & I + \underset{= 0}{\underbrace{h \sqrt{\theta} 2 E_s}} + h^2 \theta
    \underset{N}{\underbrace{(2 F_s + E^{\top} E)}} +\mathcal{O} (h^3)
  \end{eqnarray*}
  where we used that $E$ is antisymmetric. For any matrix $M$ with positive
  determinant we have the polar decomposition $M = U \sqrt{M^{\top} M} = U
  \sqrt{I + P}$, with $U \in \tmop{SO} (3)$ and $P = M^{\top} M - I$. By the
  frame invariance of the energy and a Taylor expansion around the identity of
  the square root
  \begin{eqnarray*}
    W_0 (x_3, M) & = & W_0 (x_3, \sqrt{M^{\top} M})\\
    & = & W_0 \left( x_3, I + 1 / 2  (M^{\top} M - I) + o (| M^{\top} M - I
    |) \right),
  \end{eqnarray*}
  and, assuming that a Taylor expansion of $W_0$ around the identity can be
  carried, i.e. that $M$ is close enough to SO(3), this is equal to:
  \[ \frac{1}{2} Q_3 \left( x_3, 1 / 2  (M^{\top} M - I) \right) + o (|
     M^{\top} M - I |^2) . \]
  In view of the definition of $W_0$, we set
  \[ M^h \assign \nabla_h y^h  (I + h^2  \tilde{B}^h), \]
  where $\tilde{B}^h = \sqrt{\theta} B^h \rightarrow \tilde{B} = \sqrt{\theta}
  B$ in $L^{\infty}$. Then
  \begin{eqnarray*}
    (M^h)^{\top} M^h & \assign & [\nabla_h y^h  (I + h^2  \tilde{B}^h)]^{\top}
    [\nabla_h y^h  (I + h^2  \tilde{B}^h)]\\
    & = & (I + h^2  (\tilde{B}^h)^{\top}) \nabla_h^{\top} y^h \nabla_h y^h 
    (I + h^2  \tilde{B}^h) .\\
    & = & (I + h^2  (\tilde{B}^h)^{\top})  (I + h^2 \theta N)  (I + h^2 
    \tilde{B}^h) +\mathcal{O} (h^3)\\
    & = & I + h^2 \theta N + h^2 2 \tilde{B}_s^h +\mathcal{O} (h^3)\\
    & = & I + h^2 \theta N + h^2 2 \tilde{B}_s + o (h^2) .
  \end{eqnarray*}
  To compute the first term in $h^2$, $N = 2 F_s + E^{\top} E$, we have
  \[ 2 F_s = 2 \left( \hat{\nabla}_s u - x_3 \theta^{- 1 / 2}  \hat{\nabla}^2
     v + (\partial_3 d \otimes e_3)_s \right), \]
  and:\footnote{We use the identities $(c \otimes e_3)  (e_3 \otimes c) = c
  \otimes c$, $(c \otimes e_3)  (c \otimes e_3) = c_3 c \otimes e_3$, $(e_3
  \otimes c)  (e_3 \otimes c) = c_3 e_3 \otimes c$ and $(e_3 \otimes c)  (c
  \otimes e_3) = | c |^2 e_3 \otimes e_3$.}
  \begin{eqnarray*}
    E^{\top} E & = & (\hat{\nabla} v \otimes e_3 - e_3 \otimes \hat{\nabla} v)
    (e_3 \otimes \hat{\nabla} v - \hat{\nabla} v \otimes e_3)\\
    & = & \hat{\nabla} v \otimes \hat{\nabla} v + | \hat{\nabla} v |^2 e_3
    \otimes e_3 .
  \end{eqnarray*}
  Since these quantities are independent of $h$, for sufficiently small $h$
  the product $(M^h)^{\top} M^h$ does lie close enough to $\tmop{SO} (3)$ and
  we can perform the desired Taylor expansion:
  \begin{eqnarray*}
    W^h (x_3, \nabla_h y^h) & = & W_0 (x_3, \nabla_h y^h  (I + h^2 
    \tilde{B}^h))\\
    & = & W_0 \left( x_3, ((M^h)^{\top} M^h)^{1 / 2} \right)\\
    & = & \frac{1}{2} Q_3 \left( x_3, \tfrac{1}{2}  [(M^h)^{\top} M^h - I]
    \right) + o (| (M^h)^{\top} M^h - I |^2) .
  \end{eqnarray*}
  Define now $\hat{G}_0 \assign \theta \big( \hat{\nabla}_s u + 1 / 2 
  \hat{\nabla} v \otimes \hat{\nabla} v \big)$, $\hat{G}_1 \assign -
  \theta^{1 / 2}  \hat{\nabla}^2 v$ as in Lemma
  \ref{lem:identification-limiting-strain}. Bringing the previous computations
  together we obtain:
  \begin{eqnarray*}
    \frac{1}{2}  [(M^h)^{\top} M^h - I] & = & h^2  [\hat{G}_0 - x_3  \hat{G}_1
    + \widehat{\check{\tilde{B}}}_s \nobracket\\
    &  & \left. \phantom{\check{\tilde{B}}_s} +
    \smash{\underbrace{\sqrt{\theta} (B (t)_{\cdot 3} \otimes e_3)_s +
    \tfrac{\theta}{2}  | \hat{\nabla} v |^2 e_3 \otimes e_3 + \theta
    (\partial_3 d \otimes e_3)_s}_H} \right]\\
    &  & \phantom{\check{\tilde{B}}_s} + o (h^2),
  \end{eqnarray*}
  hence
  \begin{align*}
    & \applicationspace{3 \tmop{em}} \frac{1}{h^4}  \left[ Q_3 \left( x_3, 1 / 2
    ((M^h)^{\top} M^h - I) \right) + o (| (M^h)^{\top} M^h - I |^2) \right] \\
    &\qquad =~ Q_3 \left( x_3, \hat{G}_0 - x_3  \hat{G}_1 + \sqrt{\theta} 
    \widehat{\check{B}}_s + H \right) + o (1). \applicationspace{3 \tmop{em}}
  \end{align*}
  We now choose the vector field $d$ to cancel one term and attain the minimum
  for the others by solving for $\partial_3 d$ in:
  \[ H \overset{!}{=} \left( \mathcal{L} \left( x_3, G_0 - x_3 G_1 +
     \sqrt{\theta}  \check{B}_s (x_3) \right) \otimes e_3 \right)_s, \]
  that is:
  \[ \theta^{- 1 / 2} B (t)_{\cdot 3} + \tfrac{1}{2} | \hat{\nabla} v |^2 e_3
     + \partial_3 d (x', x_3) = \tfrac{1}{\theta} \mathcal{L} \left( x_3, G_0
     - x_3 G_1 + \sqrt{\theta}  \check{B}_s (x_3) \right) . \]
  Consequently, we set:
  \begin{eqnarray*}
    d (x', x_3) & \assign & - \frac{1}{2}  | \hat{\nabla} v |^2 x_3 e_3\\
    &  & + \frac{1}{\theta}  \int_0^{x_3} \mathcal{L} \left( t, G_0 - tG_1 +
    \sqrt{\theta}  \check{B}_s (t) \right) - \sqrt{\theta} B (t)_{\cdot 3}
    \mathd t,
  \end{eqnarray*}
  and we obtain
  \[ Q_3 \left( x_3, \hat{G}_0 - x_3  \hat{G}_1 + \sqrt{\theta} B_s (x_3) + H
     \right) = Q_2 \left( x_3, G_0 - x_3 G_1 + \sqrt{\theta}  \check{B}_s
     (x_3) \right) . \]
  As in the proof of Theorem \ref{thm:upper-bound-lki}, 
  \eqref{eq:mapping-L} and \eqref{eq:Q2-L-bounds} imply that 
  $d \in W^{1, \infty} (\Omega_1 ; \mathbb{R}^3)$. 
  }}
  \smallskip 
  
  {\step{2: Convergence}{By the previous step we have $\frac{1}{\theta h^4} W_0
  (x_3, \nabla_h y^h) \rightarrow \frac{1}{2 \theta} Q_2 \left( x_3, G_0 + x_3
  G_1 + \sqrt{\theta}  \check{B}_s \right)$ a.e. as $h \rightarrow 0$, and the
  sequence is uniformly bounded so we can integrate over the domain and pass
  to the limit:
  \begin{eqnarray*}
    \frac{1}{\theta h^4}  \int_{\Omega_1} W^h (x_3, \nabla_h y^h) &
    \rightarrow & \frac{1}{2 \theta}  \int_{\Omega_1} Q_2 \left( x_3, G_0 -
    x_3 G_1 + \sqrt{\theta}  \check{B}_s \right)\\
    & = & \frac{1}{2}  \int_{\omega} \overline{Q}_2 (\theta^{1 / 2} (\grs u +
    \tfrac{1}{2} \nabla v \otimes \nabla v), - \nabla^2 v) .
  \end{eqnarray*}}}
  
  {\step{3: Convergence of the recovery sequence}{Note that 
  $P^h_{\theta} (y) \rightarrow(u, v)$ in $X_{\alpha}$ as $h \rightarrow 0$  
  with the choice $R^h = I \in \tmop{SO}(3), c^h = 0 \in \mathbb{R}^3$ in 
  Definition \ref{def:ph-maps} since 
  \begin{eqnarray*}
    \frac{1}{\theta h^2}  \int_{- 1 / 2}^{1 / 2} (y^h (\cdot, x_3)
    - x') \mathd x_3 
    & {\longrightarrow} & u \quad \text{in } W^{1,2} (\omega; \mathbb{R}^2), \\ 
    \frac{1}{\sqrt{\theta} h}  \int_{- 1 / 2}^{1 / 2} y^h_3
    (\cdot, x_3) \mathd x_3 
    & {\longrightarrow} & v \quad \text{in } W^{1,2} (\omega; \mathbb{R}). 
  \end{eqnarray*}
  }}
\end{proof}

In the next result, there is a departure from the analogous functional in
{\cite{friesecke_hierarchy_2006}} beyond the dependence on the out-of-plane
component $x_3$. In the preceding cases, if one sets $Q_2 (t, A) \equiv Q_2
(A)$, and $B \equiv 0$ then the same functionals are obtained as in that work.
However, in the regime $\alpha > 3$ their limit has no membrane term, but we
have $\overline{Q}_2 \left( \grs u, - \nabla^2 v \right) = \frac{1}{2}  \int
Q_2 \left( \grs u \right) + \frac{1}{24}  \int Q_2 (\nabla^2 v)$,
{\tmem{with}} the membrane term. The reason is that
{\cite{friesecke_hierarchy_2006}} discard the in-plane displacements by
minimising them away. In their proofs, they drop the first term in the lower
bound and build the recovery sequence with no $u$ term in $h^{\alpha - 1}$.

Note that it is by keeping the membrane term that our model is able to take
into account and respond to the pre-stressing (internal misfit) $B^h$, e.g.
compressive or tensile stresses in wafers.

\begin{theorem}[Upper bound, linearised von Kármán regime]
  \label{thm:upper-bound-lvk}Let $\alpha > 3$ and consider displacements $(u,
  v) \in X_{\alpha} \assign W^{1, 2} (\omega ; \mathbb{R}^2) \times W^{1, 2}
  (\omega ; \mathbb{R})$. There exists a sequence $(y^h)_{h > 0} \subset Y$
  which $P^h_{\alpha}$-converges to $(u, v)$ such that
  \[ \underset{h \rightarrow 0}{\lim} \mathcal{I}_{\alpha}^h (y^h)
     =\mathcal{I}_{\rm lvK} (u, v)^{}, \]
  with $\mathcal{I}_{\rm lvK}$ defined as in \eqref{eq:energy-lvk} by
  \[ \mathcal{I}_{\rm lvK} (u, v) \assign \frac{1}{2} 
     \int_{\omega} \overline{Q}_2 \left( \grs u, - \nabla^2 v \right) \mathd
     x' \]
  on $X_{\alpha}^0$ and by $+ \infty$ elsewhere.
\end{theorem}

\begin{proof}
  We follow closely the notation and path of proof of Theorem
  \ref{thm:upper-bound-vk}. By a standard density argument it is enough to
  consider $(u, v) \in X_{\alpha} \cap C^{\infty} (\overline{\omega})$. Define
  \[ y^h (x', x_3) \assign \left(\begin{array}{c}
       x'\\
       hx_3
     \end{array}\right) + \left(\begin{array}{c}
       h^{\alpha - 1} u (x')\\
       h^{\alpha - 2} v (x')
     \end{array}\right) - h^{\alpha - 1} x_3  \left(\begin{array}{c}
       \nabla v (x')\\
       0
     \end{array}\right) + h^{\alpha} d (x', x_3), \]
  with $d \in W^{1, \infty} (\Omega_1 ; \mathbb{R}^3)$. Then
  \begin{eqnarray*}
    \nabla_h y^h & = & I + h^{\alpha - 2}  \underbrace{\smash{(e_3 \otimes
    \hat{\nabla} v - \hat{\nabla} v \otimes e_3)}}_{= : E} + h^{\alpha - 1} 
    \underbrace{\smash{(\hat{\nabla} u - x_3  \hat{\nabla}^2 v + \partial_3 d
    \otimes e_3)}}_{= : F}\\
    &  & \applicationspace{1 \tmop{em}} +\mathcal{O} (h^{\alpha}),
  \end{eqnarray*}
  and, using that $E_s = 0$:
  \begin{eqnarray*}
    \nabla_h^{\top} y^h \nabla_h y^h & = & (I + h^{\alpha - 2} E^{\top} +
    h^{\alpha - 1} F^{\top})  (I + h^{\alpha - 2} E + h^{\alpha - 1} F)
    +\mathcal{O} (h^{\alpha})\\
    & = & I + 2 h^{\alpha - 1} F_s + o (h^{\alpha - 1}) .
  \end{eqnarray*}
  Define now $M^h \assign \nabla_h y^h  (I + h^{\alpha - 1} B^h)$. A few
  computations lead to
  \[ \tfrac{1}{2}  [(M^h)^{\top} M^h - I] = h^{\alpha - 1}  (F_s + B_s) + o
     (h^{\alpha - 1}), \]
  from which follows, after a Taylor approximation (recall from the proof of
  Theorem \ref{thm:upper-bound-vk}, that this can be done for sufficiently
  small $h$):
  \begin{eqnarray*}
    \frac{1}{h^{2 \alpha - 2}} W^h (x_3, \nabla_h y^h) & = & \dfrac{1}{2 h^{2
    \alpha - 2}}  [Q_3 (x_3, [(M^h)^{\top} M^h - I] / 2) \nobracket\\
    &  & \left. \applicationspace{1 \tmop{em}} + o (| (M^h)^{\top} M^h - I
    |^2) \right]\\
    & = & \dfrac{1}{2} Q_3 (x_3, F_s + B_s) + o (1) .
  \end{eqnarray*}
  Picking $d$ such that:
  \[ ((B (x_3)_{\cdot 3} + \partial_3 d) \otimes e_3)_s = (\mathcal{L} (x_3,
     \nabla u - x_3 \nabla^2 v + \check{B}_s (x_3)) \otimes e_3)_s, \]
  e.g.
  \[ d (x', x_3) \assign \int_0^{x_3} \mathcal{L} \left( t, \grs u - t
     \nabla^2 v + \check{B}_s (t) \right) - B (t)_{\cdot 3} \mathd t, \]
  the term with $\mathcal{L}$ in $Q_2$ cancels out and we obtain
  \[ Q_3 (x_3, F_s + B_s) = Q_2 \left( x_3, \grs u - x_3 \nabla^2 v +
     \check{B}_s (x_3) \right) . \]
  Note that as proved in Theorem \ref{thm:upper-bound-vk}, the properties of
  $\mathcal{L}$ imply that the function $d \in W^{1, \infty} (\Omega_1 ;
  \mathbb{R}^3)$ so the previous computations are justified. We have therefore
  \[ \frac{1}{h^{2 \alpha - 2}} W_0 (x_3, \nabla_h y^h) \rightarrow
     \cfrac{1}{2} Q_2 \left( x_3, \grs u - x_3 \nabla^2 v + \check{B}_s (x_3)
     \right) \text{ a.e. in } \omega, \]
  and also $Q_2 (x_3, A) \lesssim | A |^2$ because $Q_3$ is in $L^{\infty}$
  (Assumption \ref{main-assumptions}.\ref{assumption:Q3-ess-bounded}). Because
  $u_i, v \in C^{\infty} (\overline{\omega})$ and $B_s \in L^{\infty}$, all
  arguments of $Q_2$ are uniformly bounded and we can apply dominated
  convergence to conclude:
  \begin{eqnarray*}
    \frac{1}{h^{2 \alpha - 2}}  \int_{\Omega_1} W_0 (x_3, \nabla_h y^h) &
    \underset{h \downarrow 0}{\longrightarrow} & \frac{1}{2}  \int_{\Omega_1}
    Q_2 \left( x_3, \grs u - x_3 \nabla^2 v + \check{B}_s (x_3) \right) \mathd
    x\\
    & = & \frac{1}{2}  \int_{\omega} \overline{Q}_2 \left( \grs u, - \nabla^2
    v \right) \mathd x' .
  \end{eqnarray*}
  Set now \ $R = I \in \tmop{SO} (3), c = 0 \in \mathbb{R}^3$ for the rigid
  transformation $\rho$ in Definition \ref{def:ph-maps}. It remains to note 
  that indeed $P^h_{\alpha} (y^h) \rightarrow (u, v)$ in $X_{\alpha}$: 
  \begin{eqnarray*}
  \frac{1}{h^{\alpha - 2}}  \int_{- 1 / 2}^{1 / 2} y^h_3
     (\cdot, x_3) \mathd x_3 
  & {\longrightarrow} & v \quad \text{in } W^{1, 2}(\omega; \mathbb{R}), \\ 
  \frac{1}{h^{\alpha - 1}}  \int_{- 1 / 2}^{1 / 2}
     (y^{h\prime} (\cdot, x_3) - x') \mathd x_3 
  & {\longrightarrow} & u  \quad \text{in } W^{1, 2}(\omega; \mathbb{R}^2), 
  \end{eqnarray*}
  and the proof is complete.
\end{proof}

%-----------------------------------------------------------------------------------------
%-----------------------------------------------------------------------------------------
%-----------------------------------------------------------------------------------------
\section{$\Gamma$-convergence of the interpolating
theory}\label{sec:gamma-convergence-interpolation}
%-----------------------------------------------------------------------------------------

\begin{notation*}
  Throughout this section we write $A_{\theta} \assign \grs u_{\theta} +
  \tfrac{1}{2} \nabla v_{\theta} \otimes \nabla v_{\theta}$ for the strain
  induced by a pair of displacements $(u_{\theta}, v_{\theta})$. As before,
  $\theta > 0$.
\end{notation*}

We now set to prove Theorem \ref{thm:gamma-interpolating}, which states that
the functional of generalised von Kármán type that we found in the preceding
section,
\[ \mathcal{I}^{\theta}_{\rm vK} (u_{\theta}, v_{\theta}) \assign
   \frac{1}{2}  \int_{\omega} \int_{- 1 / 2}^{1 / 2} Q_2 \left( x_3,
   \sqrt{\theta} A_{\theta} - x_3 \nabla^2 v_{\theta} + \check{B} (x_3)
   \right) \mathd x_3 \mathd x', \]
interpolates between the two adjacent regimes as $\theta \rightarrow \infty$
or $\theta \rightarrow 0$. As $\theta$ approaches infinity, we expect the
optimal energy configurations to approach those of the linearised Kirchhoff
model, whereas with $\theta$ tending to zero they should approach the
linearised von Kármán model.

For this section we restrict ourselves to spaces where Korn-Poincaré
type inequalities hold.

\begin{definition}
  \label{def:xu-xv-spaces}Let
  \[ X_u \assign \left\{ u \in W^{1, 2} (\omega ; \mathbb{R}^2) :
     \int_{\omega} \gra u = 0 \text{ and } \int_{\omega} u = 0 \right\}, \]
  and
  \[ X_v \assign \left\{ v \in W^{2, 2} (\omega ; \mathbb{R}) : \int_{\omega}
     \nabla v = 0 \text{ and } \int_{\omega} v = 0 \right\} . \]
  We set $X_w \assign X_u \times X_v$ with the weak topologies.
\end{definition}

Additionally, from now on we assume without loss that the barycenter of 
$\omega$ be the origin so that $\int_{\omega} x' \mathd x' = 0$. Finally, 
for the limit $\theta \rightarrow \infty$ we require that {\em $\omega$ be 
convex} and recall the definition of the space of maps with singular Hessian
\[ W^{2, 2}_{s \nospace h} (\omega) \assign \left\{ v \in W^{2, 2} (\omega ;
   \mathbb{R}) : \det \nabla^2 v = 0 \text{ a.e.} \right\} . \]
\begin{remark}
  \label{rem:xu-xv-spaces-energy-same}There is no loss of generality in 
  reducing to the space $X_u \times X_v$: First we can always add an 
  infinitesimal rigid motion to $u$ and and any affine function to $v$ 
  without changing $\grs u$ or $\nabla^2 v$. Second, although the nonlinear 
  term $\nabla v \otimes \nabla v$ {\em does} change after adding an affine function, the extra terms
  appearing happen to be a symmetric gradient which can be absorbed into $\grs
  u$ with a little help: For any $g (x) = a \cdot x + b$ for $a, b \in
  \mathbb{R}^2$, we have
  \begin{eqnarray}
    \nabla (v + g) \otimes \nabla (v + g) & = & \nabla v \otimes \nabla v + a
    \otimes a + a \otimes \nabla v + \nabla v \otimes a \nonumber\\
    & = & \nabla v \otimes \nabla v + \grs z 
    \label{eq:shifted-v-nonlinear-term}
  \end{eqnarray}
  where we set $z (x) \assign (2 v (x) + a \cdot x) a \in W^{2, 2} (\omega ;
  \mathbb{R}^2)$. Therefore, for any fixed $u \in W^{1, 2} (\omega ;
  \mathbb{R}^2), v \in W^{2, 2} (\omega)$ one can choose $g (x) = - [(\nabla
  v)_{\omega} \cdot x + (v)_{\omega}]$ and define
  \[ \tilde{u} = u + z + r, \applicationspace{1 \tmop{em}} \tilde{v} = v + g,
  \]
  with $r (x) = Rx + c$, for constants $R \assign \frac{- 1}{| \omega |} 
  \int_{\omega} \nabla_a u + \nabla_a z \mathd x \in \mathbb{R}^{2 \times
  2}_{\tmop{ant}}$ and $c \assign \frac{- 1}{| \omega |}  \int_{\omega} u (x)
  + z (x) + Rx \mathd x$. For $\tilde{u}, \tilde{v}$ we then have on the one
  hand $\int \tilde{u} = 0, \int \nabla_a  \tilde{u} = 0$ and $\int \tilde{v}
  = 0, \int \nabla \tilde{v} = 0$ and on the other (note that $\grs r = 0$):
  \[ I (u, v) = I (\tilde{u} - z - r, \tilde{v} - g)
     \overset{\text{{\eqref{eq:shifted-v-nonlinear-term}}}}{=} I (\tilde{u} -
     r, \tilde{v}) = I (\tilde{u}, \tilde{v}) \]
  as desired.
\end{remark}

Our first theorem identifies the types of convergence required in order to
obtain precompactness of sequences of bounded energy. We use these definitions
of convergence for the computation of the $\Gamma$-limits.

\begin{theorem}[Compactness]
  \label{thm:compactness-interpolating}Let $(u_{\theta}, v_{\theta})_{\theta >
  0}$ be a sequence in $X_w$ with finite energy
  \[ \underset{\theta > 0}{\sup} \mathcal{I}_{\rm vK}^{\theta}
     (u_{\theta}, v_{\theta}) \leqslant C. \]
  Then:
  \begin{enumerate}
    \item The sequence $(v_{\theta})_{\theta \uparrow \infty}$ is weakly
    precompact in $W^{2, 2} (\omega)$ and the weak limit is in $X_v \cap W^{2,
    2}_{s \nospace h} (\omega)$. Additionally $(u_{\theta})_{\theta \uparrow
    \infty}$ is weakly precompact in $W^{1, 2} (\omega ; \mathbb{R}^2)$.
    
    \item The sequence $(\theta^{1 / 2} u_{\theta}, v_{\theta})_{\theta
    \downarrow 0}$ is weakly precompact in $W^{1, 2} (\omega ; \mathbb{R}^2)
    \times W^{2, 2} (\omega)$ and the weak limit is in $X_u \times X_v$.
  \end{enumerate}
\end{theorem}

\begin{proof}
  By assumption:
  \[ C \geqslant \int_{\omega} \int_{- 1 / 2}^{1 / 2} Q_2 \left( x_3,
     \sqrt{\theta} A_{\theta} - x_3 \nabla^2 v_{\theta} + \check{B} (x_3)
     \right) \mathd x_3 \mathd x', \]
  and the uniform lower bound on $Q_2$ in \eqref{eq:Q2-L-bounds} yields 
  \[ Q_2 (x_3, F) \gtrsim | F |^2 \text{ for all symmetric } F \text{ and }
     x_3 \in \left( - 1 / 2, 1 / 2 \right), \]
  so that $\int_{- 1 / 2}^{1 / 2} Q_2 (x_3, F (x_3)) \gtrsim \int_{- 1 / 2}^{1
  / 2} | F (x_3) |^2$. Now split the inner integral in half, and normalise to
  use Jensen's inequality. In the upper half:
  \begin{eqnarray*}
    C & \geqslant & \int_{\omega} 2 \int_0^{1 / 2} Q_2 \left( x_3,
    \sqrt{\theta} A_{\theta} - x_3 \nabla^2 v_{\theta} + \check{B}_s (x_3)
    \right) \mathd x_3 \mathd x'\\
    & \gtrsim & \int_{\omega} 2 \int_0^{1 / 2} \left| \sqrt{\theta}
    A_{\theta} - x_3 \nabla^2 v_{\theta} + \check{B}_s (x_3) \right|^2 \mathd
    x_3 \mathd x'\\
    & \gtrsim & \int_{\omega} \left| 2 \int_0^{1 / 2} \sqrt{\theta}
    A_{\theta} - x_3 \nabla^2 v_{\theta} + \check{B}_s (x_3) \mathd x_3
    \right|^2 \mathd x'\\
    & = & \int_{\omega} \left| \sqrt{\theta} A_{\theta} - \tfrac{1}{4}
    \nabla^2 v_{\theta} + c \right|^2 \mathd x'\\
    & \gtrsim & \left\| \sqrt{\theta} A_{\theta} - \tfrac{1}{4} \nabla^2
    v_{\theta} \right\|_{0, 2}^2 - c^2  | \omega | .
  \end{eqnarray*}
  An analogous computation for the lower half of the interval results in
  \[ C \geqslant \left\| \sqrt{\theta} A_{\theta} + \tfrac{1}{4} \nabla^2
     v_{\theta} \right\|_{0, 2} \]
  and bringing both bounds together we obtain:
  \begin{equation}
    \label{eq:thm:compactness-interpolating-1} \left\| \sqrt{\theta}
    A_{\theta} \right\|_{0, 2} \leqslant C \text{\quad and\quad} \| \nabla^2
    v_{\theta} \|_{0, 2} \leqslant C.
  \end{equation}
  Two applications of Poincaré's inequality to the second bound yield:
  \[ \| v_{\theta} \|_{2, 2} \leqslant C \text{ for all } \theta > 0. \]
  Therefore a subsequence (not relabelled) $v_{\theta} \rightharpoonup v$ for
  some $v \in X_v$. Now consider {\eqref{eq:thm:compactness-interpolating-1}}
  again and observe that with the Sobolev embedding $W^{1, 2} (\omega)
  \hookrightarrow L^4 (\omega)$ we know that
  \[ \| \nabla v_{\theta} \otimes \nabla v_{\theta} \|_{0, 2} = \| \nabla
     v_{\theta} \|^2_{0, 4} \lesssim \| \nabla v_{\theta} \|_{1, 2}^2
     \leqslant \| v_{\theta} \|^2_{2, 2} \leqslant C. \]
  Together with \eqref{eq:thm:compactness-interpolating-1} this implies   
  \begin{equation}
  \label{eq:thm:compactness-interpolating-2}
  \left\|{\sqrt{{\theta}}} {\grs}
  u_\theta \right\|_{0,2}{\leqslant}C+C {\sqrt{{\theta}}},
  \end{equation}
  so, by the Korn-Poincaré inequality, the sequence
  $(u_{\theta})_{\theta > 0}$ is bounded in $W^{1, 2}$ when $\theta
  \rightarrow \infty$ and there exists a subsequence (not relabelled)
  $u_{\theta} \rightharpoonup u$ for some $u \in X_u$.
  
  Now if $z_{\varepsilon} \rightharpoonup z$ in $W^{1, 2} (\omega ;
  \mathbb{R}^2)$, by the compact Sobolev embedding $W^{1, 2} \hookrightarrow
  L^4$ we have $z_{\varepsilon} \rightarrow z$ in $L^4$ and
  \begin{eqnarray*}
    \int_{\omega} | z_{\varepsilon} \otimes z_{\varepsilon} - z \otimes z |^2
    \mathd x \underset{\varepsilon \rightarrow 0}{\longrightarrow} 0.
    %\label{eq:nonlinear-membrane-strain-convergence}
  \end{eqnarray*}
  So $\nabla v_{\theta} \otimes \nabla v_{\theta} \rightarrow \nabla v \otimes
  \nabla v$ in $L^2$ and from {\eqref{eq:thm:compactness-interpolating-1}} and
  lower semicontinuity of the norm we deduce
  \[ \left\| \grs u + \tfrac{1}{2} \nabla v \otimes \nabla v \right\|_{0, 2}
     \leqslant \underset{\theta \rightarrow \infty}{\tmop{linf}}  \|
     A_{\theta} \|_{0, 2} = 0. \]
  By {\cite[Proposition 9]{friesecke_hierarchy_2006}} $v \in W^{2, 2}_{s
  \nospace h} (\omega)$ since $\omega$ is convex, and this concludes the proof 
  of the first statement.
  
  For the second statement we take $\theta \downarrow 0$. It only remains to
  prove precompactness for $u_{\theta}$ since the previous computation for
  $(v_{\theta})_{\theta > 0}$ applies for all $\theta$. But it follows
  directly from {\eqref{eq:thm:compactness-interpolating-2}} above: again with
  the Korn-Poincaré inequality, the sequence $(\theta^{1 / 2}
  u_{\theta})_{\theta > 0}$ is bounded in $W^{1, 2}$, so it contains a weakly
  convergent subsequence $\theta^{1 / 2} u_{\theta} \rightharpoonup u \in X_u$.
\end{proof}

We begin the proof of $\Gamma$-convergence in Theorem
\ref{thm:gamma-interpolating} with the lower and upper bound and a few
technical lemmas for the passage from $\alpha = 3$ to $\alpha < 3$.

\begin{theorem}[Lower bound, von Kármán to linearised Kirchhoff]
  \label{thm:lower-bound-vk-to-lki}Assume $\omega$ is convex and let
  $(u_{\theta}, v_{\theta})_{\theta > 0}$ be a sequence in $X_w$ such that
  $v_{\theta} \rightharpoonup v$ in $X_v$ as $\theta \rightarrow \infty$. Then
  \[ \underset{\theta \uparrow \infty}{\tmop{linf}} \mathcal{I}_{\rm vK}^{\theta} 
  (u_{\theta}, v_{\theta}) \geqslant \mathcal{I}_{\rm lKi} (v) . \]
\end{theorem}

\begin{proof}
  By Theorem \ref{thm:compactness-interpolating} we only need to consider 
  $v \in X_v^0 \assign X_v \cap W^{2,2}_{s \nospace h} (\omega)$, hence 
  $\mathcal{I}_{\rm lKi}(v) < \infty$. We can minimise the 
  inner integral pointwise and obtain a lower bound:
  \begin{eqnarray*}
    \mathcal{I}_{\rm vK}^{\theta} (u_{\theta}, v_{\theta}) & = &
    \frac{1}{2}  \int_{\omega} \int_{- 1 / 2}^{1 / 2} Q_2 \left( x_3,
    \sqrt{\theta} A_{\theta} - x_3 \nabla^2 v_{\theta} + \check{B} (x_3)
    \right) \mathd x_3 \mathd x'\\
    & \geqslant & \frac{1}{2}  \int_{\omega} \underset{A \in \mathbb{R}^{2
    \times 2}}{\min}  \int_{- 1 / 2}^{1 / 2} Q_2 (x_3, A - x_3 \nabla^2
    v_{\theta} + \check{B} (x_3)) \mathd x_3 \mathd x'\\
    & \overset{}{=} & \mathcal{I}_{\rm lKi} (v_{\theta}) .
  \end{eqnarray*}
  As $\overline{Q}_2^{\star}$ is a convex quadratic form, we have by the
  convergence $\nabla^2 v_{\theta} \rightharpoonup \nabla^2 v$ in $L^2$:
  \[ \underset{\theta \uparrow \infty}{\tmop{linf}} \mathcal{I}_{\rm vK}^{\theta} 
  (u_{\theta}, v_{\theta}) \geqslant \underset{\theta \uparrow
     \infty}{\tmop{linf}} \mathcal{I}_{\rm lKi} (v_{\theta})
     \geqslant \mathcal{I}_{\rm lKi} (v) . \]
\end{proof}

\begin{theorem}[Upper bound, von Kármán to linearised Kirchhoff]
  \label{thm:upper-bound-vk-to-lki}Assume $\omega$ is convex. Set $X^0_v
  \assign X_v \cap W^{2, 2}_{s \nospace h} (\omega)$ and fix some displacement
  $v \in X_v$. There exists a sequence $(u_{\theta}, v_{\theta})_{\theta
  \uparrow \infty} \subset X_w$ such that $v_{\theta} \rightharpoonup v$ in
  $W^{2, 2} (\omega)$ and $\mathcal{I}_{\rm vK}^{\theta} (u_{\theta},
  v_{\theta}) \rightarrow \mathcal{I}_{\rm lKi} (v)$ as
  $\theta \rightarrow \infty$.
\end{theorem}

\begin{proof}
  By Theorem \ref{thm:density-singular-hessian} we can work with functions $v
  \in \mathcal{V}_0$, see {\eqref{def:V0}}, which are smooth with singular
  Hessian, since they are dense in the restriction to $X_v$. By
  {\cite[Proposition 9]{friesecke_hierarchy_2006}} there exists a
  displacement $u : \omega \rightarrow \mathbb{R}^2$ in $W^{2, 2} (\omega ;
  \mathbb{R}^2)$ such that
  
  {\begin{equation}
  \label{eq:thm:upper-bound-vk-to-lki}
  {\grs} u+{\tfrac{1}{2}} {\nabla}v{\otimes}{\nabla}v=0.
  \end{equation}}
  
  Fix $\delta > 0$ and, using Corollary
  \ref{cor:representation-matrix-minimizer}, choose smooth functions $\alpha
  \in C^{\infty} (\overline{\omega}), g \in C^{\infty} (\overline{\omega} ;
  \mathbb{R}^2)$ such that
  \[ \left\| \grs g + \alpha \nabla^2 v - A_{\min} \right\|^2_{0, 2} < \delta,
  \]
  where $A_{\min} \in L^{\infty} (\omega ; \mathbb{R}^{2 \times
  2}_{\tmop{sym}})$ is defined as
  \[ A_{\min} \assign \underset{A \in \mathbb{R}^{2 \times
     2}_{\tmop{sym}}}{\tmop{argmin}} \int_{- 1 / 2}^{1 / 2} Q_2 (t, A - t
     \nabla^2 v + \check{B} (t)) \mathd t. \]
  Define now the recovery sequence $(u_{\theta}, v_{\theta})_{\theta > 0}$
  with
  \[ u_{\theta} \assign u + \tfrac{1}{\sqrt{\theta}}  (\alpha \nabla v + g),
     \applicationspace{1 \tmop{em}} v_{\theta} \assign v -
     \tfrac{1}{\sqrt{\theta}} \alpha . \]
  Clearly $v_{\theta} = v - \theta^{- 1 / 2} \alpha \rightarrow v$ as $\theta
  \rightarrow \infty$ in $W^{2, 2} (\omega)$. Furthermore
  \[ \sqrt{\theta}  \grs u_{\theta} = \sqrt{\theta}  \grs u + \grs g + (\nabla
     \alpha \otimes \nabla v)_s + \alpha \nabla^2 v \]
  \[ \tfrac{\sqrt{\theta}}{2} \nabla v_{\theta} \otimes \nabla v_{\theta} =
     \tfrac{\sqrt{\theta}}{2} \nabla v \otimes \nabla v + \tfrac{1}{2
     \sqrt{\theta}} \nabla \alpha \otimes \nabla \alpha - (\nabla \alpha
     \otimes \nabla v)_s, \]
  and 
  \[ - t \nabla^2 v_{\theta} = - t \nabla^2 v + \tfrac{t}{\sqrt{\theta}}
     \nabla^2 \alpha, \]
  so that, using {\eqref{eq:thm:upper-bound-vk-to-lki}} and the fact that the
  product $\| \nabla \alpha \otimes \nabla \alpha \|_{0, 2} = \| \nabla \alpha
  \|_{0, 4}^2$ is bounded we have
  \begin{eqnarray*}
    \mathcal{I}_{\rm vK}^{\theta} (u_{\theta}, v_{\theta}) & = &
    \frac{1}{2}  \int_{\omega} \int_{- 1 / 2}^{1 / 2} Q_2 \left( t, \theta^{1
    / 2} A_{\theta} - t \nabla^2 v_{\theta} + \check{B} (t) \right) \mathd t
    \mathd x'\\
    & = & \frac{1}{2}  \int_{\omega} \int_{- 1 / 2}^{1 / 2} Q_2 \left( t,
    \grs g + (\alpha - t) \nabla^2 v + \check{B} (t) \right) \mathd t \mathd
    x' +\mathcal{O} ( \theta^{- 1 / 2} ) .
  \end{eqnarray*}
  Now subtract and add $A_{\min}$ inside $Q_2$ and use Cauchy's inequality
  to get 
  \begin{align*}
    & \int_{- 1 / 2}^{1 / 2} Q_2 \left( t, \grs g + \alpha \nabla^2 v - t \nabla^2 v +
    \check{B} \right) \mathd t \\
    &~~ \leqslant~ \left( 1 + \sqrt{\delta} \right)  \int_{- 1 / 2}^{1 / 2} Q_2 (t,
    A_{\min} - t \nabla^2 v + \check{B}) \mathd t\\
    &\qquad~ + \frac{1}{4 \sqrt{\delta}}  \underbrace{\int_{- 1 / 2}^{1 / 2} Q_2
    \left( t, \grs g + \alpha \nabla^2 v - A_{\min} \right) \mathd
    t}_{\lesssim \left\| \grs g + \alpha \nabla^2 v - A_{\min} \right\|^2_{0,
    2} < \delta}\\
    &~~ =~ \int_{- 1 / 2}^{1 / 2} Q_2 (t, A_{\min} - t \nabla^2 v + \check{B})
    \mathd t +\mathcal{O}_{\delta \downarrow 0} ( \delta^{1 / 2} ).
  \end{align*}
  We plug this in and obtain:
  \begin{eqnarray*}
    \mathcal{I}_{\rm vK}^{\theta} (u_{\theta}, v_{\theta}) & \leqslant &
    \frac{1}{2}  \int_{\omega} \int_{- 1 / 2}^{1 / 2} Q_2 (t, A_{\min} - t
    \nabla^2 v + \check{B} (t)) \mathd t \mathd x'\\
    &  & \applicationspace{2 \tmop{em}} +\mathcal{O} ( \theta^{- 1 / 2} ) 
    + \mathcal{O}_{\delta \downarrow 0} ( \delta^{1 / 2} )\\
    & \overset{\theta \uparrow \infty}{\longrightarrow} & \frac{1}{2} 
    \int_{\omega} \int_{- 1 / 2}^{1 / 2} Q_2 (t, A_{\min} - t \nabla^2 v +
    \check{B} (t)) \mathd t \mathd x' +\mathcal{O}_{\delta \downarrow 0}
    ( \delta^{1 / 2} ) .
  \end{eqnarray*}
  The proof is concluded by letting $\delta \rightarrow 0$ and passing to a
  diagonal sequence.
\end{proof}

We finish the proof of Theorem \ref{thm:gamma-interpolating} with the lower
and upper bounds for the transition from $\alpha = 3$ to $\alpha > 3$. The
lack of constraints in the limit functional makes the proofs straightforward.

\begin{theorem}[Lower bound, von Kármán to linearised von Kármán]
  \label{thm:lower-bound-vk-to-lvk}Let $(u_{\theta}, v_{\theta})_{\theta > 0}$
  be a sequence in $X_w$ such that $(\theta^{1 / 2} u_{\theta}, v_{\theta})
  \rightarrow (u, v)$ in $X_w$ as $\theta \rightarrow 0$. Then
  \[ \underset{\theta \rightarrow 0}{\tmop{linf}} \mathcal{I}_{\rm vK}^{\theta} 
  (u_{\theta}, v_{\theta}) \geqslant \mathcal{I}_{\rm lvK} (u, v) . \]
\end{theorem}

\begin{proof}
  We may assume that $\sup_{\theta > 0} \mathcal{I}^{\theta}_{\rm vK}
  (u_{\theta}, v_{\theta}) \leqslant C$. Then by Theorem
  \ref{thm:compactness-interpolating} $(\nabla v_{\theta})_{\theta > 0}$ is
  bounded in $W^{1, 2}$ and by the Sobolev embedding $W^{1, 2} \hookrightarrow
  L^4$ we have as before $\| \nabla v_{\theta} \otimes \nabla v_{\theta}
  \|_{0, 2} = \| \nabla v_{\theta} \|_{0, 4}^2 \leqslant C$. Consequently
  \[ \sqrt{\theta} A_{\theta} = \sqrt[]{\theta}  \grs u_{\theta} +
     \frac{\sqrt{\theta}}{2} \nabla v_{\theta} \otimes \nabla v_{\theta}
     \rightharpoonup \grs u \text{\quad in } L^2 \text{ as } \theta \downarrow
     0. \]
  By convexity of the quadratic form $Q_2$ we conclude 
  \begin{eqnarray*}
    \underset{\theta \downarrow 0}{\tmop{linf}} \mathcal{I}_{\rm vK}^{\theta} 
  (u_{\theta}, v_{\theta}) & \geqslant & \frac{1}{2} 
    \int_{\omega} \int_{- 1 / 2}^{1 / 2} Q_2 \left( x_3, \grs u - x_3 \nabla^2
    v + \check{B} (x_3) \right) \mathd x_3 \mathd x'\\
    & = & \mathcal{I}_{\rm lvK} (u, v) .
  \end{eqnarray*}
\end{proof}

\begin{theorem}[Upper bound, von Kármán to linearised von Kármán]
  \label{thm:upper-bound-vk-to-lvk}Let $(u, v) \in X_w$. There exists a
  sequence $(u_{\theta}, v_{\theta})_{\theta > 0} \subset X_w$ converging to
  $(u, v) \in X_w$ such that $\mathcal{I}_{\rm vK}^{\theta} (u_{\theta},
  v_{\theta}) \rightarrow \mathcal{I}_{\rm lvK} (u, v)$ as
  $\theta \rightarrow 0$.
\end{theorem}

\begin{proof}
  Define
  \[ u_{\theta} \assign \theta^{- 1 / 2} u \text{\quad and\quad} v_{\theta}
     \assign v. \]
  Clearly $(\theta^{1 / 2} u_{\theta}, v_{\theta}) \equiv (u, v)$ and using
  again $W^{1, 2} \hookrightarrow L^4$ we have:
  \[ \sqrt{\theta} A_{\theta} = \grs u + \tfrac{1}{2} \theta^{1 / 2} \nabla v
     \otimes \nabla v \underset{\theta \downarrow 0}{\longrightarrow}  \grs u
     \text{\quad in } L^2 . \]
  Consequently:
  \begin{eqnarray*}
    \mathcal{I}_{\rm vK}^{\theta} (u_{\theta}, v_{\theta}) & = &
    \frac{1}{2}  \int_{\omega} \int_{- 1 / 2}^{1 / 2} Q_2 \left( x_3,
    \sqrt{\theta} A_{\theta} - x_3 \nabla^2 v_{\theta} + \check{B} (x_3)
    \right) \mathd x_3 \mathd x'\\
    & \underset{\theta \downarrow 0}{\longrightarrow} & \frac{1}{2} 
    \int_{\omega} \int_{- 1 / 2}^{1 / 2} Q_2 \left( x_3, \grs u - x_3 \nabla^2
    v + \check{B} (x_3) \right) \mathd x_3 \mathd x'\\
    & \overset{}{=} & \mathcal{I}_{\rm lvK} (u, v),
  \end{eqnarray*}
  as stated.
\end{proof}

%-----------------------------------------------------------------------------------------
%-----------------------------------------------------------------------------------------
%-----------------------------------------------------------------------------------------
\section{Approximation and representation
theorems}\label{sec:approximation-and-representation}
%-----------------------------------------------------------------------------------------

A key ingredient in the proofs of the upper bounds is the density of certain
smooth functions in the space where the energy is minimised. In particular,
for the case $\alpha \in (2, 3)$ we obtain a result proving that $W^{2, 2}$
maps with singular Hessian can be approximated by a specific set of smooth 
functions with the 
same property. In order to apply the results of {\cite{schmidt_plate_2007}} we
may restrict ourselves to isometries which partition $\omega$ into finitely
many so-called {\tmem{bodies}} and {\tmem{arms}}. More precisely, 
suppose $y : \omega \to \mathbb{R}^3$ is a $W^{2, 2}$ isometric immersion and denote 
by $\tmop{II} = \tmop{II}_{(y)}$ its second fundamental form, i.e., 
$\tmop{II}_{ij} = y_{,i}\cdot (y_{,1}\wedge y_{,2})_{,j}$. Then $\tmop{II}$ 
is singular, and there exists $f_y\in W^{1,2}$ such that $\nabla f_u = \tmop{II}$. 
We call $\gamma:[0,l]\to \omega$, parameterised by arclength, a {\em leading curve} 
if it is orthogonal to the inverse images of $f_y$ on regions where $f_y$ is not 
constant. We denote by $\kappa$ and $\nu$ the curvature and unit normal, respectively, 
i.e., $\gamma'' = \kappa\nu$. In fact, $\kappa$ must be bounded, hence 
$\gamma \in W^{2,\infty}$. A subdomain $\omega'\subset \omega$ is said to be 
{\em covered} by a curve $\gamma$ if
\[ \omega'\subset\{\gamma(t) + s\nu(t): s\in \mathbb{R}, t\in[0,l]\}. \] 
As shown in {\cite{pakzad_sobolev_2004}}, $\omega$ can be partitioned into so-called bodies and 
arms. Here a {\em body} is a connected maximal subdomain on which $y$ is 
affine and whose boundary contains more than two segments inside $\omega$. 
An {\em arm} is a maximal subdomain $\omega(\gamma)$ covered by some 
leading curve $\gamma$. 

In {\cite{schmidt_plate_2007}} (built on {\cite{pakzad_sobolev_2004}}) 
it is shown that the set
\[ \mathcal{A}_0 \assign \left\{ y \in C^{\infty} (\overline{\omega} ;
   \mathbb{R}^3  \nobracket : y \text{ is an isometry finitely partitioning }
   \omega \right\}, \]
is dense in the $W^{2, 2}$-isometries. Here we show that, additionally,
\begin{equation}
  \label{def:V0} \mathcal{V}_0 \assign \{v \in C^{\infty} (\overline{\omega})
  : \exists \eta > 0 \text{ s.t. } \eta v = y_3 \text{ for some } y \in
  \mathcal{A}_0 \}
\end{equation}
is $W^{2, 2}$-dense in $W^{2, 2}_{s \nospace h}$.\footnote{The density of $C^{2}(\omega) \cap 
  W^{2, 2}_{s \nospace h}(\omega)$ in $W^{2, 2}_{s \nospace h}(\omega)$ was 
  first announced in {\cite{pakzad_sobolev_2004}} to follow along the 
  same lines as the density of smooth isometric immersions in the class of 
  $W^{2,2}$ isometric immersions. As this seems not to be straightforward, we 
  follow a different route reducing the density of $\mathcal{V}_0$ in 
  $W^{2, 2}_{s \nospace h}$ to the density of $\mathcal{A}_0$ in the set of 
  $W^{2,2}$ isometric immersions. We are grateful to Peter Hornung for the help provided with this
  proof.}

\begin{theorem}
  \label{thm:density-singular-hessian}Let $\omega \subset \mathbb{R}^2$ be a
  bounded, convex, Lipschitz domain. Then the set $\mathcal{V}_0$ is $W^{2,
  2}$-dense in $W^{2, 2}_{s h} (\omega)$. In particular $\det \nabla^2 v = 0$
  for all $v \in \mathcal{V}_0$.
\end{theorem}

\begin{proof}
   
  {\step{1: Approximation}{Let $v \in W^{2, 2}_{s \nospace h} (\omega)$ and
  $\varepsilon > 0$. By {\cite[Theorem 10]{friesecke_hierarchy_2006}}, we
  can find some $\tilde{v} \in W^{2, 2}_{s \nospace h} (\omega) \cap W^{1,
  \infty} (\omega)$ s.t. $\| v - \tilde{v} \|_{2, 2} < \varepsilon / 2$ and,
  for $\eta = \eta (\varepsilon) > 0$ sufficiently small, $\| \nabla \eta
  \tilde{v} \|_{\infty} < 1 / 2$. One can now apply {\cite[Theorem
  7]{friesecke_hierarchy_2006}} to construct an isometry $\tilde{y} \in
  W^{2, 2} (\omega ; \mathbb{R}^3)$ whose out-of-plane component $\tilde{y}_3
  = \eta \tilde{v}$. By {\cite[Proposition
  2.3]{schmidt_plate_2007}}  
  we find a smooth $y \in \mathcal{A}_0$ such that 
  $\| y - \tilde{y} \|_{2, 2} < \varepsilon \eta / 2$ and in particular 
  $\| y_3 - \tilde{y}_3 \|_{2, 2} < \varepsilon \eta / 2$. 
  Setting $\psi \assign y_3 / \eta \in \mathcal{V}_0$ we conclude
  \[ \| v - \psi \|_{2, 2} \leqslant \| v - \tilde{v} \|_{2, 2} + \| \tilde{v}
     - \psi \|_{2, 2} < \varepsilon . \]}}
  
  {\step{2: Inclusion}{Let $v \in \mathcal{V}_0$ with $\eta v = y_3, \eta > 0$
  for some smooth isometry $y \in \mathcal{A}_0$. Recall that the second
  fundamental form $\Iota \Iota_{(y)}$ of any smooth isometric immersion $y$
  is singular and the identity $\nabla^2 y_j = - \Iota \Iota_{(y)} n_j$ holds
  for all $j \in \{ 1, 2, 3 \}$, where $n = y_{, 1} \wedge y_{,
  2}$.\footnote{See {\cite[Proposition 3]{muller_regularity_2005}} for a
  proof for $W^{2, 2}$ isometries on Lipschitz domains.} Therefore $\det (\eta
  \nabla^2 v) = \det (- \tmop{II}_{(y)} n_3) = 0$ and the proof is complete.}}
\end{proof}

\begin{remark}
We note that the following similar statement can be proved using the same
approximation arguments and {\cite[Theorem 1]{hornung_approximation_2011}}
(with the bonus of in addition holding for more general domains). 
  Let $\omega \subset \mathbb{R}^2$ be a bounded, simply connected, Lipschitz
  domain whose boundary contains a set $\Sigma = \overline{\Sigma} \subset
  \partial \omega$ with $\mathcal{H}^1 (\Sigma) = 0$ such that on its
  complement $\partial \omega \backslash \Sigma$ the outer unit normal to
  $\omega$ exists and is continuous. Then the set $W^{2, 2}_{s h} (\omega)
  \cap C^{\infty} (\overline{\omega})$ is $W^{2, 2}$-dense in $W^{2, 2}_{s h}
  (\omega)$.
\end{remark}

Once one can work with smooth functions, the essential tool for the
construction of the recovery sequences for $\alpha \in (2, 3)$ is the
following representation theorem for maps with singular Hessian and its
corollary, both inspired by {\cite{schmidt_plate_2007}}. A crucial
component in the proof of the result in that paper is the ability to use
approximations partitioning the domain in regions over which they are affine.
This is in close connection to the {\tmem{rigidity property}} for $W^{2,
2}$-isometries proved in {\cite[Theorem II]{pakzad_sobolev_2004}}: every
point of their domain lies either on an open set or on a segment connecting
the boundaries where the map is affine.

\begin{theorem}
  \label{thm:representation-matrix}Let $v \in \mathcal{V}_0$ and $A \in
  C^{\infty} (\overline{\omega} ; \mathbb{R}^{2 \times 2}_{\tmop{sym}})$ such
  that $A \equiv 0$ in a neighbourhood of $\{ \nabla^2 v = 0 \}$. There exist
  maps $\alpha, g_1, g_2 \in C^{\infty} (\overline{\omega})$ such that $\alpha
  = g_i = 0$ on $\{ \nabla^2 v = 0 \}$ and
  \[ A = \grs g + \alpha \nabla^2 v. \]
\end{theorem}

\begin{proof}
  Let $\eta > 0, y \in \mathcal{A}_0$ s.t. $\eta v = y_3$. Using that
  $\nabla^2 y_3 = - \Iota \Iota_{(y)} n_3$ holds by virtue of $y$ being an
  isometry, with $n = y_{, 1} \wedge y_{, 2}$ being the unit normal vector, we
  have that $A \equiv 0$ in a neighbourhood of $\{ \tmop{II}_{(y)} = 0 \} \cup
  \{ n_3 = 0 \}$, and
  \[ \{ \nabla^2 v = 0 \} = \{ \tmop{II}_{(y)} = 0 \} \cup \{ n_3 = 0 \} . \]
  We can apply {\cite[Lemma 3.3]{schmidt_plate_2007}}\footnote{Namely: If $y
  \in \mathcal{A}_0$ and $A \in C^{\infty} (\overline{\omega} ; \mathbb{R}^{2
  \times 2}_{\tmop{sym}})$ vanishes over a neighbourhood of $N = 
  \{ \tmop{II}_{(y)} = 0 \}$, then there exist $\tilde{\alpha}, g_1, g_2 \in C^{\infty}
  (\overline{\omega})$ vanishing on $N$ such that $A = \grs g + \tilde{\alpha} \tmop{II}_{(y)}$.} 
  to $y$ in order to obtain functions $\tilde{\alpha}, g_1, g_2 \in
  C^{\infty} (\overline{\omega})$ s.t. $\tilde{\alpha}, g_1, g_2 = 0$ on 
  $\{ \tmop{II}_{(y)} = 0 \}$ and $A = \grs g + \tilde{\alpha} \tmop{II}_{(y)}$.
  
  By examining the proof of this Lemma one can see that $\tilde{\alpha}, g
  \equiv 0$ in a neighbourhood of $\left\{n_3 = 0\right\}$: since over bodies
  one has $\tilde{\alpha}, g_1, g_2 = 0$ by construction, we need only
  consider arms. On these sets, if $n_3$ vanishes at a point then it
  vanishes at a whole line perpendicular to the leading curve, because the
  latter is orthogonal to the level sets of the gradient. Now, because $A = 0$
  in a neighbourhood of this line, when solving the equations in the proof of
  the Lemma which determine $g$ then $\tilde{\alpha}$, one obtains $u_{2, s} =
  0$ and $u_{2, t} = 0$, and with the boundary conditions $u_2 = 0$ then $u_1
  = 0$ is a solution to the remaining equation. Hence $g = 0$ and
  $\tilde{\alpha} = 0$ on these lines. Since the functions so obtained are
  $C^{\infty}$, we can define $\alpha \assign - \tilde{\alpha} \eta / n_3$ if
  $n_3 \neq 0$ and $\alpha = 0$ otherwise, and this is a smooth function such
  that
  \[ A = \grs g + \alpha \nabla^2 v. \]
\end{proof}

\begin{corollary} 
  \label{cor:representation-matrix-minimizer}Let $v \in \mathcal{V}_0$ and
  define for every $x' \in \omega$
  \[ A_{\min} (x') = \underset{A \in \mathbb{R}^{2 \times
     2}_{\tmop{sym}}}{\tmop{argmin}} \int_{- 1 / 2}^{1 / 2} Q_2 (t, A - t
     \nabla^2 v (x') + \check{B}_s) \mathd t. \]
  Then $A_{\min} \in L^2 (\omega ; \mathbb{R}^{2 \times 2}_{\tmop{sym}})$ and
  there exist sequences of functions $\alpha_k \in C^{\infty} (\overline{\omega}), 
  g_k \in C^{\infty} (\overline{\omega} ; \mathbb{R}^2)$
  such that
  \[ \left\| \grs g_k + \alpha_k \nabla^2 v - A_{\min} \right\|_{L^2 (\omega ;
     \mathbb{R}^{2 \times 2})} \longrightarrow 0 \text{ as } k \rightarrow
     \infty . \]
\end{corollary}

\begin{proof}
  Let $k \in \mathbb{N}$ be arbitrary. First, on the set $\{ \nabla^2 v = 0
  \}$ we trivially have $A_{\min} \equiv A_0$ a constant matrix. Now let $A_k
  \in C^{\infty} (\overline{\omega} ; \mathbb{R}^{2 \times 2})$ with support  
  in $\{ \nabla^2 v \neq 0 \}$ such that
  \[ \| A_k - (A_{\min} - A_0) \|_{L^2 (\omega ; \mathbb{R}^{2 \times 2})} <
     \tfrac{1}{k} \]
  and use Theorem \ref{thm:representation-matrix} to pick smooth $\alpha_k,
  \tilde{g}_k$ on $\overline{\omega}$ with 
  \[ A_k = \grs  \tilde{g}_k + \alpha_k \nabla^2 v. \]
  Set $g_k (x') = \tilde{g}_k (x') + A_0 x'$. Then:
  \[ \left\| \grs g_k + \alpha_k \nabla^2 v - A_{\min} \right\|_{L^2} =
     \left\| \grs  \tilde{g}_k + \alpha_k \nabla^2 v - (A_{\min} - A_0)
     \right\|_{L^2} < \tfrac{1}{k} .  \]
\end{proof}

%-----------------------------------------------------------------------------------------
\section*{Acknowledgements}

We are grateful to Peter Hornung for the help provided with the proof of Theorem~\ref{thm:density-singular-hessian}. This work was financially supported by project 285722765 of the Deutsche Forschungsgemeinschaft (DFG, German Research Foundation), ``{\tmem{Effektive Theorien und Energie minimierende Konfigurationen für
heterogene Schichten}}''.

%-----------------------------------------------------------------------------------------

\bibliography{thesis} 
\bibliographystyle{abbrv}

\end{document}